\documentclass[a4paper,12pt]{article}
\usepackage{latexsym}
\usepackage{graphics} % for pdf, bitmapped graphics files
\usepackage{epsfig} % for postscript graphics files
\usepackage{amsmath} %for mathbb
\usepackage{amsfonts, amssymb} %for mathbb
\usepackage{amsthm,enumerate}
\usepackage{color}
\usepackage{mathrsfs}

\theoremstyle{theorem}
\newtheorem{thm}{Theorem}[section]
\newtheorem{prop}[thm]{Proposition}
\newtheorem{lemma}[thm]{Lemma}
\newtheorem{coro}[thm]{Corollary}

\theoremstyle{definition}

\newtheorem{rem}[thm]{Remark}

\newtheorem{ass}[thm]{Assumption}

\numberwithin{equation}{section}

\newcommand{\bi}[1]{\ensuremath{\boldsymbol{#1}}}   %%% ボールドイタリック %%%%

\begin{document}
\title{\LARGE{\bf A Discrete-Time Clark-Ocone Formula  
and its Application to an Error Analysis 
}}
\author{
Jir\^o Akahori\footnote{This work was supported by JSPS KAKENHI Grant Numbers 
23330109, 24340022, 23654056 and 25285102.},
Takafumi Amaba\footnote{This work was supported by JSPS KAKENHI Grant Number 24$\cdot$5772}
and Kaori Okuma \\
%Kaori Okuma \\
\\[8pt]
Department of Mathematical Sciences，Ritsumeikan University \\
1-1-1 Nojihigashi，Kusatsu，Shiga, 525-8577，Japan \\
%E-mail: kaori.okuma@gmail.com
}

\date{June 2013}
\maketitle
%\thispagestyle{empty}
%ABSTRACT
\abstract{
In this paper, we will establish 
a discrete-time version of Clark(-Ocone-Haussmann) formula, which can be 
seen as an asymptotic expansion in a weak sense.  
The formula is 
applied to the estimation of the error caused by the 
martingale representation.  
In the way, we use another distribution theory 
with respect to Gaussian rather than Lebesgue measure, 
which can be seen as a discrete Malliavin calculus.  
}
%\tableofcontents

\section{Introduction}
Let $ T > 0 $, $(W_t)_{0 \leq t \leq T}) $ be a Brownian motion
starting from $ 0 $, and 
$ (\mathcal{F}_t)_{0 \leq t \leq T} $ be its natural filtration. 
Let $ X \in L^2 (\mathcal{F}_T) $ be differentiable 
in the sense of Malliavin, for which we may write 
$ X \in \mathbb{D}_{2,1}$ (see e.g.\cite{IW}).
Then, it holds that 
\begin{equation}\label{Th:CO}
X=E[X]+\int_{0}^{T} E[D_{s}X|{\mathcal{F}}_{s}]dW_s,
\end{equation}
where $D_{s}$ means the Malliavin derivative (evaluated at $ s $). 

The formula (\ref{Th:CO})
is known as Clark-Ocone formula  
though there are many variants; Clark \cite{clark} obtained (\ref{Th:CO})
for some Fr\'echet differentiable functionals and 
Ocone \cite{ocone} %,\cite{ocone karatzas} 
related it to Malliavin derivatives, while 
Haussmann \cite{Haussmann2} extended it to functionals of a solution
to a stochastic differential equation. 
There are yet much more contexts, which we omit here. 

In the context of mathematical finance, 
the formula gives an alternative description of 
the hedging portfolio in terms of Malliavin derivatives. 
However, explicit expressions of 
the Malliavin derivatives 
of a Wiener functional are not available in general 
(except for some special cases: see \cite{jean}).
In the paper we will introduce a {\em finite dimensional approximation} of
(\ref{Th:CO}) and discuss the ``order of the convergence"
in a finance-oriented mode\footnote{Actually, this kind of finite-dimensional 
approximation or something similar is commonly used in financial practice. 
Hence the results presented in this paper 
might be more insightful and useful for the practitioners in the field.}. 

Let us be more precise. 
Put $ \Delta W_k =W_{k \Delta t} - W_{(k-1)\Delta t} $
for $ k \in \mathbb{N} $, where $ \Delta t $ %$ T> 0 $ 
is a fixed constant. 
Then, for fixed $ n $, the random variable 
$ (\Delta W_1,\cdots, \Delta W_n) $  
is distributed as $ N (0, \Delta t I ) $. 
Let $ \mathcal{G}_k $, $ k=1, \cdots N $,
be the sigma-algebra generated by
$ (\Delta W_1, \cdots, \Delta W_k) $. 
Note that $ \mathbb{G}:=\{ \mathcal{G}_k \}_{k=0}^{N} $ 
is a filtration, and 
\begin{equation*}
L^2 (\mathcal{G}_{N} , P) \simeq L^{2}(\mathbb{R}^{N},\mu ^{N}), 
\end{equation*}
where 
\begin{equation*}
\mu^{N}(dx) = \frac{1}{{(2\pi \Delta t)}^{\frac{N}{2}}}
e^{-\frac{{\left|x\right|}^{2}}{2\Delta t}}dx. 
\end{equation*}

With the filtration $ \mathbb{G} $, we can discuss 
``stochastic integral" (which is in fact a Riemannian sum)
with respect to 
the process (random walk) $ W^{\Delta t} = \sum \Delta W $.
On the other hand, we can naturally define (a precise formulation
will  be given in section \ref{GWF}) 
a finite dimensional version of the Malliavin derivative $ D_s $
by the weak partial derivatives such as 
\begin{equation*}
\partial_l X ( x_1 , \cdots , x_N ) |_{x_k = \Delta W_k, k=1, \cdots, N}. 
\end{equation*}
Then one might well guess that 
a discrete version of the Clark-Ocone formula could be
\begin{equation*}
X \overset{?}{=} \mathbf{E} [X] + \sum_{l=1}^{N} 
\mathbf{E} [\partial_l X |\mathcal{G}_{l-1}] \Delta W_l %^{\Delta t}
\end{equation*}
but this is not true since the random walk $ W^{\Delta t} $
does not have the martingale representation property\footnote{
If the martingale representation property holds for a random walk, 
then we can establish a precise discrete-time Clark-Ocone formula
if we define ``differentiation" properly. For the binary case, 
N. Privault \cite{Privault1} has made a detailed study on 
the discrete Clark-Ocone formula and related discrete 
Malliavin calculus. }. 

We should instead ask how much the (martingale representation) error,  
\begin{equation*}
\mathrm{Mart.Err}:=  
X -\mathbf{E} [X] - \sum_{l=1}^{N} \mathbf{E} [\partial_l X |\mathcal{G}_{l-1}] \Delta W_l, 
\end{equation*}
measured by a norm, 
(in this paper we concentrate on the estimation 
with respect to $ L^2 (\mathbb{R}^{N}) $-one), is. 
Further, %in view of applications to finance, we also study 
its asymptotic behaviour as $ N \to \infty $ with $ N \Delta t = \text{time horizon $T$} $. 
This is closely related to the problem of so-called {\em tracking error 
of the delta hedge}. 
If one has a nice finite dimensional approximation $ X^{N} $ of 
a Wiener functional $ X $, both defined on the same probability space, 
then the tracking error can be controlled 
by the (supremum in $ N $ of) $ \mathrm{Mart.Err} $ 
plus the error caused by the discretization (finite-dimensional approximation)
as we see from
%\footnote{The decomposition can also be 
%applied to the error caused by some discretization schemes of Backward differential equations 
%(\cite{Zhang}, \cite{Bouchard-Touzi}, \cite{Gobet-Labar}
%\cite{Gobet-Azmi1}, etc), and thus the topic 
%is also closely related to our ``discretization of Clark formula".
%The connection will be separately studied in another paper.}
:
\begin{equation*}
\begin{split}
\mathrm{Tra.Err}&:= 
X - \mathbf{E} [X] - \sum_{l=1}^{N} \mathbf{E} [\partial_{l \Delta t}
X |\mathcal{F}_{l\Delta t} ] \Delta W_l \\
%\right\Vert_{L^2 (P)}
& = X- X^{N} + \mathbf{E} [ X - X^{N}] \\
& - \sum_{l=1}^{N} \left( 
\mathbf{E} [D_{(l \Delta t)}
X |\mathcal{F}_{(l \Delta t)} ] 
-\mathbf{E} [\partial_l X^{N} |\mathcal{G}_{l-1}] \right)
\Delta W_l + \mathrm{Mart.Err} \\
&=: \mathrm{Disc. Err} + \mathrm{Mart.Err}. 
\end{split}
\end{equation*}

There are considerably many studies on the subject of the tracking error
as well. 
It at least dates back to the paper by Rootzen \cite{Roo}, 
where the weak convergence of the scaled error was studied.    
The problem is reformulated as ``tracking error of the delta hedge" 
in Bertsimas, Kogan, and Lo \cite{BKL}, where the error was also 
measured by $ L^2 $-norm. Hayashi and Mykland \cite{Hayashi-Mykland} further developed the argument from financial perspectives. 

%Roughly speaking, the known 
Notable results in this topic are summarized as follows. 
\begin{enumerate}
\item The scaled tracking error $ N^{-1/2} \mathrm{Tra.Err} $ converges
weakly to $ B_\tau $ with\footnote{Here actually the differentiability is not required. The expression 
$ \mathbf{E}[D_s X|\mathcal{F}_s] $ should be understood as 
simply the integrand of the martingale representation of $ X $ and 
the meaning of $ \mathbf{E}[D^2_s X|\mathcal{F}_s] $ will be clarified later.}
$$ \tau = \frac{1}{2} \int |\mathbf{E}[D^2_s X|\mathcal{F}_s]|^2 ds, $$
where $ B $ is a Brownian motion independent of $ \tau $. 

\item The tracking error estimated with $ L^2 $-norm is in $ O ( N^{-1/2} ) $
in the cases of $ X = F (S) $ with ``ordinary pay-off" $ F $ and
the solution $ S $ of an SDE, while it is in $ O ( N^{-1/4} ) $ when
$ F $ is ``irregular"like Heaviside function (Gobet and Temam \cite{Gobet-Tenam}, Temam \cite{Temam}). 
Later the irregularity is associated with differentiability in the {\em fractional order} $ s \in (0,1)$ by Geiss and Geiss \cite{GeGe};
it is in $ O ( N^{-s/2} ) $ for $ s $-differentiable $ F $. 

%\item Geiss and Geiss \cite{GeGe} further revealed that 
%the order can be improved by taking proper time discretization 
%which is not necessarily uniform. 
%Actually, it can be always made $ O ( N^{-1/2} ) $-convergent. (嘘かも)
\end{enumerate}

In this paper, we shall establish the corresponding results 
for the $ \mathrm{Mart.Err} $, which %are in fact %completely 
almost parallel with the above.  

After introducing the Discrete Clark-Ocone formula (Theorem \ref{ACO2},
section \ref{MW}), we will show, by using the formula, 
a multi-level central limit theorem 
for the error (Theorem \ref{WEAK}). This corresponds to the result 1 above. 
Since we will be working on a \underline{sequence}
of {\em discrete} Wiener functionals 
unlike the situations concerning tracking error, 
we need to some discussions on the finite-dimensionality. 
An answer is given in section \ref{ASYMP}, and under the condition
it is proven that the convergence order is related to a
fractional smoothness(Theorem \ref{ErrEst}). 
This corresponds to the result 2 above. 
Section \ref{ADDsec} is devoted to a study of 
the asymptotics of the error of the additive functionals. 
As a case study, we give a detailed estimate of the 
martingale representation error of the Riemann-sum approximation 
of Brownian occupation time (Theorem \ref{Occ}). 

The proofs given in this paper is largely based on elementary calculus 
with a bit of 
classical Fourier analysis and distribution theory, but nonetheless 
our methods can be, in spirit, a finite-dimensional reduction of 
Malliavin-Watanabe's distribution
theory. Some detailed discussions on this point of view will be given
in sections \ref{GWF}, \ref{Comm}, and \ref{Consis}. 
%(see e.g. \cite[Chapter V-8]{IW}). 
%as we shall see in section \ref{MW}. 
We have restricted ourselves to one-dimensional Wiener space case, 
but this is only for simplicity for the notations. 

\section{A Discrete Version of Clark-Ocone Formula}
\subsection{Generalized Wiener Functional in Discrete Time}\label{GWF}
Throughout this section we fix $ N \in \mathbb{N} $ and 
work on the canonical probability space 
$ (\mathbb{R}^{N}, \mathfrak{B}(\mathbb{R}^{N}), \mu^{N} ) $
though we will abuse the notations like $ \Delta W $ 
as the coordinate map. 

Let $ \mathcal{S}_{N} \equiv \mathcal{S} (\mathbb{R}^{N} ) $ be 
the Schwartz space; the space of all rapidly decreasing functions
and $ \mathcal{S}'_{N} $ be its dual; the space of all tempered distributions
(see, e.g. \cite{Ru}). We (may) call $ X \in \mathcal{S}'_{N} $
a ``discrete generalized Wiener functional" 
and its generalized expectation 
is defined to be the coupling 
$ {}_{\mathcal{S}'_{N}} \langle X, p^{N} \rangle_{\mathcal{S}_{N}} $,
where $ p^{N} $ is the density of $ \mu^{N} $,
which is of course in $ \mathcal{S}_{N} $. 

The conditional expectation $ \mathbf{E} [X|\mathcal{G}_k] $
for $ X \in \mathcal{S}'_{N} $ is then defined as follows.
We first note that the inclusion 
$ \mathcal{G}_k \subset \mathcal{G}_{N} $ 
induces those of $ \mathcal{S} (\mathbb{R}^k) \
\subset \mathcal{S} (\mathbb{R}^{N} ) $ and
$ \mathcal{S}' (\mathbb{R}^k) \subset \mathcal{S}' (\mathbb{R}^{N} ) $. 
In this sense we write $ \mathcal{S}_k $ and $ \mathcal{S}'_k $
for the Schwartz space and the space of generalized
Wiener functionals with respect to $ \mathcal{G}_k $, $k=1, \cdots, N $.
Then 
$ Y =  \mathbf{E} [X|\mathcal{G}_k] $ in $ \mathcal{S}'_k $ is
defined in terms of the relation
\begin{equation*}
\mathbf{E} [ X Z ] = \mathbf{E}[ Y Z ],\quad \forall Z \in \mathcal{S}_k, 
\end{equation*}
which should be understood as 
\begin{equation*}
{}_{ \mathcal{S}'_{N} } \langle X, Z p^{N} \rangle_{ \mathcal{S}_{N} } 
= {}_{\mathcal{S}'_k} \langle Y, Z p^k 
\rangle_{\mathcal{S}_k},\quad \forall Z \in \mathcal{S}_k.
\end{equation*}
In particular, we see that the conditional expectation is well-defined
by du Bois-Reymond lemma (see e.g. \cite{Ru}). 
Note that this generalized conditional expectation reduces
to the standard one on $ L^1 ( \mu^{N} ) $, 
which is included in $ \mathcal{S}'_{N} $ unlike the $ L^1 $ space
with respect to the Lebesgue measure.
Furthermore, differentiations of $ X \in \mathcal{S}'_{N} $ 
are defined as usual, namely, 
\begin{equation*}
\partial_k X = Y \iff {}_{ \mathcal{S}'_{N} } 
\langle Y, Z \rangle_{ \mathcal{S}_{N} } 
= -  {}_{ \mathcal{S}'_{N} } 
\langle X, \partial_k Z \rangle_{ \mathcal{S}_{N} }
\quad \forall Z \in \mathcal{S}_{N},
\end{equation*}
which imply
\begin{equation*}
E [\partial_k X] = E [X \partial_k \log p^N ],
\end{equation*}
and so on.

\subsection{Clark-Ocone Formula in Discrete Time}\label{MW}
We have the following series expansion in $ \Delta t $:
\begin{thm}[A Discrete Version of Clark-Ocone Formula]\label{ACO2}%%%%
For $ X \in L^2 ( \mathcal{G}_{N} ) \simeq L^2 ( \mu^{N} ) $,
we have the following $ L^2 $-convergent series expansion:
\begin{equation}\label{AE1}
X - E[X]= \sum_{m=1}^\infty \sum_{l=1}^{N}
\frac{(\Delta t)^{m/2}}{\sqrt{m!}} 
\mathbf{E} [\partial^m_l X|\mathcal{G}_{l-1}] 
H_m \left(\frac{\Delta W_l}{\sqrt{\Delta t}}\right)
\end{equation}\label{Her}
where $ H_m $ is the $ m $-th Hermite polynomial for
$ m \in \mathbb{Z}_+ $; 
\begin{equation}
H_{m}(x)=\frac{(-1)^m}{\sqrt{m!}}e^{\frac{x^{2}}{2}}
\frac{d^{m}}{dx^{m}}e^{-\frac{x^{2}}{2}}\quad (m\in \mathbb{Z}_+). 
\end{equation}
Here the differentiations are understood in the distribution sense, 
as explained in the previous section. 
\end{thm}%%%%%%%%%%%%%%%%%%%%%%%%%%%%%%%%%%%%%%%%%%%%%%%%%%%%%%%%%%%%%
\begin{proof}%%%%%%%%%%%%%%%%%%%%%%%%%%%%%%%%%%%%%%%%%%%%%%%%%%%%%%%%%
Since 
$
\Big\{
\prod _{i=1}^{N} H_{{k}_{i}}(\frac{\Delta {W_{i}}}{\sqrt{\Delta t}})
\Big\} _{k_{1},\dots,k_{n}\in \mathbb{Z}_+}
$
is an orthonormal basis of $ L^{2} ( \mathbb{R}^{N}, \mu^{N} ) $, 
we have the following orthogonal expansion of 
$ X \in L^{2}(\mathbb{R}^{N}, \mu^{N}) $:
\begin{equation}\label{exp1}
X(\Delta W_{1},\dots,\Delta W_{N})\\
= \sum _{k_{1},\dots,k_{N}} %\langle X, \prod _{i=1}^{n}H_{{k}_{i}} \rangle 
c_{(k_{1},\dots,k_{N})}
\prod _{i=1}^{N}
H_{{k}_{i}}
\left( \frac{\Delta {W_{i}}}{\sqrt{\Delta t}} \right). 
\end{equation}
where we denote 
\begin{equation*}
c_{ (k_{1},\dots ,k_{N}) }
:=
\big\langle
	X,
	\prod _{i=1}^{N} H_{{k}_{i}} \left( \frac{ \Delta W_{i} }{\sqrt{\Delta t}} \right)
\big\rangle
=
\mathbf{E} \big[
	X
	\prod _{i=1}^{N} H_{{k}_{i}}
	\left( \frac{ \Delta W_{i} }{\sqrt{\Delta t}} \right)
\big].
\end{equation*}
Let us ``sort" the series according as the ``highest" non-zero $ k_i $;
\begin{multline}\label{PCO2}
X ( \Delta W_{1}, \dots , \Delta W_{N} ) \\
=E[X]+
%\sum _{l=1}^{n} \sum _{k_{1},\dots,k_{l-1}} c_{(k_{1},\dots,k_{l-1},1,0,\dots,0)} \prod _{i=1}^{l-1}H_{{k}_{i}}(\frac{\Delta {W_{i}}}{\sqrt{\Delta t}})\cdot
%H_{1}(\frac{\Delta {W_{l}}}{\sqrt{\Delta t}}) 
\sum _{l=1}^{N} \sum _{k_{1},\dots,k_{l-1}} \sum _{k_{l} \geq 1} c_{(k_{1},\dots , k_{l}, 0, \dots , 0)} 
\prod _{i=1}^{l}H_{{k}_{i}}
\left(\frac{\Delta {W_{i}}}{\sqrt{\Delta t}}\right).
\end{multline}

Here we claim that 
\begin{equation}\label{cond1}
\begin{split}
& \sum _{l=1}^{N} 
\sum _{k_{1},\dots,k_{l-1}} c_{(k_{1},\dots , k_{l}, 0, \dots , 0)} 
\prod _{i=1}^{l-1}H_{{k}_{i}}\left(\frac{\Delta {W_{i}}}{\sqrt{\Delta t}}
\right) \\
& \qquad \qquad 
= \mathbf{E} \left[ X H_{k_l} \left(\frac{\Delta {W_{i}}}{\sqrt{\Delta t}}
\right) \bigg| \mathcal{G}_{l-1}
\right]. 
\end{split}
\end{equation}
In fact, from the expansion (\ref{exp1}) we have 
\begin{equation*}
\begin{split}
&
\mathbf{E}[ X H_{k_l} \left( \frac{ \Delta W_{l} }{ \sqrt{\Delta t} } \right) |\mathcal{G}_{l-1}] \\
&=
\mathbf{E}\big[ \sum_{k'_{1}, \dots , k'_{N}}
c_{(k'_{1},\dots,k'_{N})}
H_{k_l} \left( \frac{ \Delta W_{l} }{ \sqrt{\Delta t} } \right)
\prod_{i=1}^{N} H_{ k_{i}^{\prime} } \left( \frac{ \Delta W_{i} }{ \sqrt{\Delta t} } \right)
\big| \mathcal{G}_{l-1} \big] \\
&= \sum_{k'_{1},\cdots, k'_N;} 
c_{(k'_{1},\dots,k'_{N}) }
\prod _{i=1}^{l-1}
H_{ k_{i}^{\prime} } \left( \frac{ \Delta W_{i} }{ \sqrt{\Delta t} } \right)
\mathbf{E} \big[
H_{k_l} \left( \frac{ \Delta W_{l} }{ \sqrt{\Delta t} } \right)
\prod _{i=l}^{N} H_{ k_{i}^{\prime} } \left( \frac{ \Delta W_{i} }{ \sqrt{\Delta t} } \right)
\big],
\end{split}
\end{equation*}
and we confirm the claim since 
$ 
\mathbf{E}[ 
H_{k_l} ( \frac{ \Delta W_{l} }{\sqrt{\Delta t}} )
\prod _{i=l}^{n} H_{ k_{i}^{\prime} } ( \frac{ \Delta W_{i} }{\sqrt{\Delta t}} )
]= 0
$ %\end{equation*}
unless $ k'_l= k_l $ and $ k'_i = 0 $ for $ i > l $. 

We further claim that
\begin{equation} \label{DIBP}
\mathbf{E} \left[ X H_{k_l} \left(\frac{\Delta {W_{i}}}{\sqrt{\Delta t}}
\right) \bigg| \mathcal{G}_{l-1}
\right] = \frac{(\Delta)^{k/2}}{\sqrt{k!}}
\mathbf{E} \left[ \partial_l^k X \big| \mathcal{G}_{l-1}
\right],
\end{equation}
which, together with (\ref{PCO2}) and (\ref{cond1}),
will prove the expansion (\ref{AE1}) in the $ L^2 $ case.
Here, 
the conditional expectation %in the right-hand-side 
should be understood in the generalized sense.
%but since the right-hand-side is the classical one, 
Following the definition we have made, it suffices to show that
\begin{equation*}
\mathbf{E}
\left[
	X
	H_{k_l} \left(\frac{\Delta {W_{l}}}{\sqrt{\Delta t}} \right)
	f ( \Delta W_1 , \cdots , \Delta W_{l-1} )
\right]
=
\frac{ (\Delta t)^{k/2} }{ \sqrt{k!} }
\mathbf{E}[ \partial_l^k X f ]
\end{equation*}
for any $ f \in \mathcal{S}_{l-1} $ but this is easy to see 
if we write down the generalized expectation as 
the coupling of $ \mathcal{S} $ and $ \mathcal{S}'$:
\begin{equation*}
\begin{split}
{}_{\mathcal{S}'} \langle X  , H_k ( x / \sqrt{\Delta t} ) f
p^{N} \rangle_{\mathcal{S}}
&= {}_{\mathcal{S}'} \langle X , 
f (-1)^k \frac{ (\Delta t)^{k/2} }{ \sqrt{k!} }
\partial_l^k p^{N} \rangle_{\mathcal{S}} \\
&= \frac{ (\Delta t)^{k/2} }{ \sqrt{k!} }
{}_{\mathcal{S}'} \langle \partial
^k_l X, f p^{N} \rangle_{\mathcal{S}}.
\end{split}
\end{equation*} 
\end{proof}%%%%%%%%%%%%%%%%%%%%%%%%%%%%%%%%%%%%%%%%%%%%%%%%%%%%%%%%%%%

%%%%%%%%%%%%%%%%%%%%%%%%%%%%%%%%%%%%%%%%%%%%%%%%%%%%%%%%%%%%%%%%%%%%%%
\subsection{Comment on Discrete Generalized Wiener Functionals}\label{Comm}
%%%%%%%%%%%%%%%%%%%%%%%%%%%%%%%%%%%%%%%%%%%%%%%%%%%%%%%%%%%%%%%%%%%%%%
In this subsection, we remark that our discrete generalized 
Wiener functionals is slightly broader than that of the direct 
finite dimensional reduction; %of the generalized Wiener functionals
there is a gap. For simplicity, we let $ \Delta t = 1$ in this subsection. 

We know that (see e.g. \cite[Appendix to V.3]{ReedSimon})  
the orthogonal expansion in $ L^2 (\mathbb{R}^{N}, \mathrm{Leb}) $
with respect to the Hermite functions:
\begin{equation*}
\phi_{N} (x) := \frac{1}{\sqrt{N!}} H_{N} (x%/\sqrt{\Delta t}
) ( p^N )^{1/2} 
\end{equation*}
gives so-called 
$ \mathcal{N} $-representation of $ \mathcal{S} $ and $ \mathcal{S}'$;
the series for $ f \in \mathcal{S} $  (resp. 
$ \in \mathcal{S}' $)
\begin{equation*}
\sum {}_{\mathcal{S}'} \langle f, \phi_N \rangle_{\mathcal{S}} \phi_N
\end{equation*}
converges to $ f $ in $ \mathcal{S} $ (resp. 
in $ \mathcal{S}' $). 
In our context, it then follows that 
if $ X ( p^N )^{1/2} \in \mathcal{S} $ (resp. 
$ \in \mathcal{S}' $), then the convergence of the expansion (\ref{AE1}) 
is in $ \mathcal{S} $ (resp. in $ \mathcal{S}' $) as well. 
It should be further noted that we have the following equivalences:
\begin{prop}
It holds that 
\begin{equation}\label{DDinfty}
X ( p^N )^{1/2} \in \mathcal{S}
\iff
X \in \mathbb{D}^{(N)}_{2,\infty} 
= \cap_{s>0} \mathbb{D}^{(N)}_{2,s} 
\end{equation}
and 
\begin{equation}\label{DGWF}
X ( p^N )^{1/2} \in \mathcal{S}'
\iff
X \in \mathbb{D}^{(N)}_{2,-\infty} 
= \cup_{s<0} \mathbb{D}^{(N)}_{2,s},
\end{equation}
where $ \mathbb{D}^{(N)}_{2,s} $ is the completion of $ L^2 ( \mu^N ) $
by the norm 
$ \Vert f \Vert_{2,s} = \Vert (1+L)^{s/2} f \Vert_{ L^2 ( \mu^N ) } $. Here
$ L $ is the Ornstein-Uhlenbeck operator on $ \mathbb{R}^N $;
\begin{equation*}
L =- \sum_{i=1}^N \frac{\partial^2}{\partial x_i^2 } 
+ \sum_{i=1}^N x_i \frac{\partial}{\partial x_i}. 
\end{equation*}
\end{prop}
\begin{proof}
Let $ \{\phi_n : n \in \mathbb{Z} \} $ be norms defined by
\begin{equation*}
\phi_n (f) = \Vert (1+S)^n f \Vert_{L^2 (\mathrm{Leb})},
\end{equation*}
where $ S $ is the following 
Schr\"odinger operator of the harmonic oscillator:
\begin{equation*}
S := - \sum_{i=1}^N \frac{\partial^2}{\partial x_i^2 } + \frac{1}{4}
%\frac{1}{2} 
|x|^2 - \frac{1}{2}. 
\end{equation*}
We know that $ \mathcal{S} $ is a Fr\'echet space 
by the seminorms $ \{ \phi_n \} $. 
In fact, both $ L $ and $ S $ are the number operators respectively
in that;
\begin{equation*}
L \prod_{i=1}^N H_{k_i} (x_i) 
= (\sum_{i=1}^N k_i ) \prod_{i=1}^N H_{k_i} (x_i) 
\end{equation*}
and 
\begin{equation*}
S \prod_{i=1}^N \phi_{k_i} (x_i) 
= (\sum_{i=1}^N k_i ) \prod_{i=1}^N \phi_{k_i} (x_i). 
\end{equation*}
We also have 
\begin{equation*}
L(f) ( p^N )^{1/2} = S (f ( p^N )^{1/2}), 
\end{equation*}
which implies 
\begin{equation*}
\Vert f \Vert_{2,n} = \phi_n (f). 
\end{equation*}
This proves (\ref{DDinfty}). 

The equivalence (\ref{DGWF}) follows from the 
following equivalence of the duality:
\begin{equation*}
{}_{\mathbb{D}^{(N)}_{2, -\infty}}\langle X, 
Y \rangle_{\mathbb{D}^{(N)}_{2, \infty}}
= {}_{\mathcal{S}'} 
\langle X ( p^N )^{1/2}, Y ( p^N )^{1/2} \rangle_{\mathcal{S}}. 
\end{equation*}
\end{proof}

\begin{coro}
For $ X \in \mathbb{D}^{(N)}_{2,s} $, $ s \in \mathbb{R} $, 
the convergence of (\ref{AE1}) 
is also attained in $ \mathbb{D}^{(N)}_{2,s} $.
\end{coro}
\begin{proof}
It follows from the fact that, by the assumption, the partial sums 
\begin{equation*}
X_n := \sum_{ k_1 + \cdots +k_N \leq n } c_{ ( k_1, \cdots, k_N ) }
\prod_{i=1}^N H_{k_i} (x_i), \quad n \in \mathbb{N} 
\end{equation*}
form a Cauchy sequence in $ \mathbb{D}^{(N)}_{2,s} $. 
\end{proof}

%%%%%%%%%%%%%%%%%%%%%%%%%%%%%%%%%%%%%%%%%%%%%%%%%%%%%%%%%%%%%%%%%%%%%%%%%%%%%%%%
\section{Asymptotic Analysis of Martingale Representation Errors}%%%%%%%%%%%%%%%%%%%%%%%%%%%%%%%%%%%%%%%%%%%%%%%%%%%
%%%%%%%%%%%%%%%%%%%%%%%%%%%%%%%%%%%%%%%%%%%%%%%%%%%%%%%%%%%%%%%%%%%%%%%%%%%%%%%%
%Let $W=(W_{t})_{0\leq t\leq T}$ be a one-dimensional Brownian motion.
%Let $(\mathcal{G}_{t})_{0\leq t \leq T}$ be the filtration generated by $W$.
In this section, we will consider the asymptotic behavior
of the error term when $ N \to \infty $ with $ N \Delta t = T $.
For this purpose, 
to make explicit 
the dependence on $ N $ 
we redefine some of the notations. 
$t_{k}:=:t_{k}^{(N)} := \frac{kT}{N}$ for each $k=0,1,\cdots ,N$.
We also write $ \Delta W^N_k = W_{t^{(N)}_k} - W_{t^{(N)}_{k-1}} $ for 
each $ k $ and $ N $, and $ \mathcal{G}^N_k 
:= \sigma (\Delta W^N_l; l=1, \cdots,k) $. 
Further, to facilitate the discussion in the limit, 
we implement our discrete Malliavin-Watanabe calculus into 
the classical one in the first subsection.

\if0%@@@@@@@@@@@@@@@@@@@@@@@@@ COMMENTED OUT BEGIN @@@@@@@@@
We further assume that $ n = 2^N $ for some $ N $. Then
$ \{ \mathcal{G}^n:= \sigma ( W_{k/n}; k=1, \cdots, n ) ; n \in \mathbb{N} \} $
forms a filtration.
\fi%@@@@@@@@@@@@@@@@@@@@@@@@@@ COMMENTED OUT END @@@@@@@@@@@

%%%%%%%%%%%%%%%%%%%%%%%%%%%%%%%%%%%%%%%%%%%%%%%%%%%%%%%%%%%%%%%%%%%%%%%%%%%%%%%%
\subsection{Consistency with the Classical Malliavin Calculus}\label{Consis}%%
%%%%%%%%%%%%%%%%%%%%%%%%%%%%%%%%%%%%%%%%%%%%%%%%%%%%%%%%%%%%%%%%%%%%%%%%%%%%%%%%
First, we review briefly the Malliavin calculus
over the one-dimensional classical Wiener space to introduce notations
which we will use in the following sections devoted to asymptotic analyses,
and then will show how our framework, 
established in the previous sections,
is ``embedded" to the classical Malliavin calculus (Proposition \ref{COM}).

Let $(\mathscr{W}, \mathbf{P})$ be the one-dimensional Wiener space on $[0,T]$.
We consider the canonical process $w=(w(t))_{0\leq t\leq T}$ starting from zero a.s.
In this context, the %Sobolev 
Hilbert
space
$$
H = \Big\{
h \in \mathscr{W}:
\begin{array}{l}
\text{$h(0)=0$ and $h$ is absolutely continuous} \\
\text{with square-integrable derivative}
\end{array}
\Big\}
$$
equipped with the inner product defined by
$$
\langle h_{1}, h_{2} \rangle_{H}
=
\int_{0}^{T} \dot{h}_{1}(t) \dot{h}_{2}(t) dt,
\quad h_{1}, h_{2} \in H
$$
is called the {\it Cameron-Martin subspace} of $\mathscr{W}$.
%The space $H$ becomes a separable Hilbert space under this inner product.
For each complete orthonormal system (CONS, in short) $\{ h_{i} \}_{i=1}^{\infty}$ of $H$,
it is known that
$\displaystyle
\big\{
\prod_{i=1}^{\infty}
H_{a_{i}} \big( \int_{0}^{T} \dot{h}_{i}(t) dw(t) \big)
: \bi{a} \in \bi{\Lambda}
\big\}
$
forms a CONS in $L^{2}(\mathscr{W})$ (see e.g., \cite{IW} Proposition 8.1), where
$\bi{\Lambda}$ is the set of all sequence $\bi{a}=(a_{i})_{i=1}^{\infty}$
of nonnegative integers except for a finite number of $i$'s and $H_{n}$ is the
$n$-th Hermite polynomial defined in (\ref{Her}).
We also denote by
$J_{n}:L^{2}(\mathscr{W}) \to \mathcal{C}_{n}$
the orthogonal projection, where $\mathcal{C}_{n}$ is the $L^{2}(\mathscr{W})$-closure of the subspace
spanned by
$\displaystyle
\big\{
\prod_{i=1}^{\infty}
H_{a_{i}} \big( \int_{0}^{T} \dot{h}_{i}(t) dw(t) \big)
: \sum_{i=1}^{\infty} a_{i} = n
\big\}
$
over $\mathbb{R}$.
Each $\mathcal{C}_{n}$ is called the subspace of 
$ n $-th {\it Wiener's homogeneous chaos}.

For each $s \in \mathbb{R}$, a Sobolev-type Hilbert space $\mathbb{D}_{2,s}=\mathbb{D}_{2,s}(\mathbb{R})$
is defined as the completion of
$
\{ F \in L^{2}(\mathscr{W}): \Vert F \Vert_{\mathbb{D}_{2,s}} < \infty \}
$
under the seminorm $\Vert \cdot  \Vert_{\mathbb{D}_{2,s}}$
on $L^{2}(\mathscr{W})$ defined by
$$
\Vert F \Vert_{\mathbb{D}_{2,s}}^{2}
=
\sum_{n=0}^{\infty} (1+n)^{s} \Vert J_{n}F \Vert_{L^{2}}^{2},
\quad F \in L^{2}(\mathscr{W})
$$
which may be infinite in general.

In the following, for any two separable Hilbert space $H_{1}$ and $H_{2}$,
we denote by $H_{1} \otimes H_{2}$ the completion of the algebraic tensor
product of $H_{1}$ and $H_{2}$ under the Hilbert-Schmidt norm.

It is known that one can define a (continuous) linear operator $D:\mathbb{D}_{2,1} \to L^{2}(\mathscr{W}) \otimes H$
such that
$$
\langle DF ,h \rangle_{H} = D_{h}F \in L^{2}(\mathscr{W})
$$
for every $h\in H$ and $F \in \mathbb{D}_{2,1}$, where $D_{h}F$ is defined by
$$
(D_{h}F)(w)
=
\lim_{\varepsilon \to 0}
\frac{1}{\varepsilon}
\Big\{
F( w + \varepsilon h ) - F(w)
\Big\}
\quad\text{for a.e. $w\in \mathscr{W}$,}
$$
which is well-defined due to the Cameron-Martin theorem (see e.g., \cite{IW} Theorem 8.5).
For each $t \in [0,T]$, let $e_{t}:\mathscr{W} \to \mathbb{R}$ denote the evaluation map
defined by $e_{t}(w) = w(t)$. Then a linear operator
$D_{t}:\mathbb{D}_{2,1} \to L^{2}(\mathscr{W})$ is defined by
\begin{equation}
D_{t}F
=
\frac{d}{dt} \big( \mathrm{id}_{L^{2}(\mathscr{W})} \otimes e_{t} \big)
( DF ),
\quad F \in \mathbb{D}_{2,1}
\label{Mal-Der1}
\end{equation}
for a.a. $t \in [0,T]$.

Under these notations, we can state the relationship between our framework
established in section 2 and that of Malliavin calculus.
We omit the proof because it is immediate from the definition.
\begin{prop}\label{COM}%%%%%%%%%%%%%%%%%%%%%%%%%%%%%%%%%%%%%
For each $X \in \mathbb{D}_{2,1}^{(N)}$, we have
$$
(D_{t} X)(w) = \sum_{l=1}^{N} 1_{ \{ t_{l-1} \leq t < t_{l} \}} ( \partial_{l} X  )(w)
$$
for a.a. $(t,w) \in [0,T] \times \mathscr{W}$.
\end{prop}%%%%%%%%%%%%%%%%%%%%%%%%%%%%%%%%%%%%%%%%%%%%%%%%%%

For each $F \in \mathbb{D}_{2,1}$, one can prove that
$\mathbf{E} [ F\vert \mathcal{G}_{N}^{N} ] \in \mathbb{D}_{2,1}^{(N)}$
and $\lim_{N\to\infty}\mathbf{E} [ F\vert \mathcal{G}_{N}^{N} ] = F$ in $\mathbb{D}_{2,1}$
(consult e.g., \cite{MaTh} Theorem 1.10).
By using also the fact that $e_{t}(h) = \langle 1_{[0,t)}, h\rangle_{H}$
for each $h \in H$, one can obtain
\begin{equation}
( D_{t}F )(w)
=
\lim_{N \to \infty}
\sum_{l=1}^{N}
1_{ \{ t_{l-1} \leq t < t_{l} \} }
\partial_{l} \mathbf{E} [ F \vert \mathcal{G}_{N}^{N} ] (w)
\label{Mal-D}
\end{equation}
for a.a. $(t,w) \in [0,T] \times \mathscr{W}$.
Note that in \cite{MaTh}, the derivative $ D $ 
on the path space $ \mathscr{W} $ 
is defined directly by (\ref{Mal-D}) with $N=2^{n}$. 
%We would like to point out that 
%Although we have employed the standard approach
%(in the current literature) in introducing the derivative in (\ref{Mal-Der1}),
%the formulation given in (\ref{Mal-D}) with $N=2^{n}$ was in fact an idea
%by Malliavin to define a derivative on the path space 
%(see \cite{MaTh}, subsection 1.3). 
Following this approach in \cite{MaTh}, 
we define $D_{\cdot}^{k} X \in L^{2}[0,T] \otimes L^{2}(\mathbf{P})$
as the $L^{2}$-limit of the sequence
$(D_{\cdot}^{k} \mathbf{E}[ X \vert \mathcal{G}_{N}^{N} ])_{N=1}^{\infty}$
if it exists (see \cite{MaTh}, Theorem 1.10 to consult what condition is enough to get this limit).

\if
we define for each $F\in\mathbb{D}_{2,n}$,
$$
( D_{t}^{n}F )(w)
=
\limsup_{N \to \infty}
\sum_{l=1}^{N}
1_{ \{ t_{l-1} \leq t < t_{l} \} }
\partial_{l}^{n} \mathbf{E} [ F \vert \mathcal{G}_{N}^{N} ] (w)
$$
for a.a. $(t,w) \in [0,T] \times \mathscr{W}$.
In the following, we will only consider 
such $F \in \mathbb{D}_{2,n}$ for which
$D_{\cdot}^{n}F$ is an element in $L^{2}[0,T] \otimes L^{2} (\mathscr{W})$.
One may consult Theorem 1.10 in \cite{MaTh} for sufficient condition for which
$D_{\cdot}^{n}F$ is in fact the $L^{2}[0,T] \otimes L^{2}(\mathscr{W})$-limit
of
$
\sum_{l=1}^{N}
1_{ [ t^{(N)}_{l-1} , t^{(N)}_{l} ) }
\partial_{l}^{n} \mathbf{E} [ F \vert \mathcal{G}_{N}^{N} ]
$.
\fi
%Finally we would like to remark that Malliavin calculus can be carried on
%general Gaussian probability spaces rather than the classical Wiener space,
%via a representation of the Heisenberg algebra (see e.g., \cite{Nu}, \cite{Privault1} etc).

By the above discussions, we may write 
$$
D_{t}^{k}X := \partial_{l}^{k} X
\quad\text{if $t_{l-1} \leq t < t_{l}$}
$$
for $X \in \mathbb{D}_{2,n}^{(N)}$, $t\in [0,T]$, and
$k=1,2,\cdots ,n $.  

\if
$ X \in
\{
	Y \in L^{2}(\mathcal{G}_{T}) :
	\mathbf{E}[ Y \vert \mathcal{G}_{N}^{N} ] \in \mathbb{D}_{2,n}^{(N)}
	\quad\text{for each $N$}
\}
$
and $k=1,2,\cdots ,n$, we define $D_{\cdot}^{k} X \in L^{2}[0,T] \otimes L^{2}(\mathbf{P})$
as the $L^{2}$-limit of the sequence
$(D_{\cdot}^{k} \mathbf{E}[ X \vert \mathcal{G}_{N}^{N} ])_{N=1}^{\infty}$
if it exists (see \cite{MaTh}, Theorem 1.10 to consult what condition is enough to get this limit).
\fi

%%%%%%%%%%%%%%%%%%%%%%%%%%%%%%%%%%%%%%%%%%%%%%%%%%%%%%%%%%%%
\subsection{A Central Limit Theorem for the Errors}%%%%%%%%%%%%%%%%%%%%%%%%%%%%%%%%%%%%%
%%%%%%%%%%%%%%%%%%%%%%%%%%%%%%%%%%%%%%%%%%%%%%%%%%%%%%%%%%%%

\if0%@@@@@@@@@@@@@@@@@@@@@@@@@ COMMENTED OUT BEGIN @@@@@@@@@
In this subsection, we specify our partition to be
$$
\Delta : 0 =t_{0} < t_{1}=\frac{T}{N} < \cdots < t_{n} = \frac{nT}{N} < \cdots t_{N}=T
$$
and write $\Delta t \equiv \Delta t_{n} = t_{n} - t_{n-1}$.
\fi%@@@@@@@@@@@@@@@@@@@@@@@@@@ COMMENTED OUT END @@@@@@@@@@@

Suppose that we are given a sequence $(X^{N})_{N=1}^{\infty}$
of finite dimensional Wiener functionals $X^{N} \in L^{2}( \mathcal{G}^{N}_N )$
for each $ N $.

We put, for $n\geq 0$, 
$$
\mathrm{Err}_{N}(n)
:=
X^{N}
-
\sum_{m=0}^{n} \sum_{l=1}^{N}
\frac{ (\Delta t)^{m/2} }{ \sqrt{m!} }
\mathbf{E} \big[
%\partial_{l}^{m} 
D_{lT/N}^m X^{N} \big\vert \mathcal{G}_{l-1}^{N}
\big]
H_{m} \left( \frac{ \Delta W^N_{l} }{ \sqrt{\Delta t} } \right) .
$$

\begin{thm}\label{WEAK}%%%%%%%%%%%%%%%%%%%%%%%%%%%%%%%%%%%%%
Let $n \in \mathbb{N}$.
Suppose that $X^{N} \in \mathbb{D}_{2,n+2}^{(N)}$ for each $N=1,2,\cdots$
and for some Wiener functional $X \in \mathbb{D}_{2,n+1}( \mathbb{R} )$, we have
\begin{itemize}
\item $X^{N} \to X$ in $L^{2}(\mathbf{P})$,

\item $\displaystyle
\int_{0}^{T} \hspace{-2mm} \Vert D_{t}^{p+1} X^{N} - D_{t}^{p+1} X \Vert_{L^{2}}^{2} dt
\to
0
$
\end{itemize}
as %$| \Delta | \to 0$ 
$ N \to \infty $ for each $p=0,1,\cdots ,n$ and
\begin{itemize}
\item $\displaystyle
\sup_{N} %=1,2,\cdots} 
\int_{0}^{T} \hspace{-2mm} \Vert D_{t}^{n+2} X^{N} \Vert_{L^{2}}^{2} dt < \infty .
$
\end{itemize}
Then we have
\begin{equation*}
\begin{split}
&
\left(\begin{array}{c}
\mathrm{Err}_{N}(0) \\
( \Delta t )^{-1/2} \mathrm{Err}_{N}(1) \\
\vdots \\
( \Delta t )^{-n/2} \mathrm{Err}_{N}(n)
\end{array}\right)
\to
\left(\begin{array}{c}
\displaystyle
\int_{0}^{T} \hspace{-2mm} \mathbf{E} \big[
D_{t} X \big\vert \mathcal{G}_{t}
\big]
dW_{t} \\
\displaystyle
\frac{1}{\sqrt{ 2 }}
\int_{0}^{T} \hspace{-2mm} \mathbf{E} \big[
D_{t}^{2} X \big\vert \mathcal{G}_{t}
\big]
dB_{t}^{1} \\
\vdots \\
\displaystyle
\frac{1}{\sqrt{ (n+1)! }}
\int_{0}^{T} \hspace{-2mm} \mathbf{E} \big[
D_{t}^{n+1} X \big\vert \mathcal{G}_{t}
\big]
dB_{t}^{n}
\end{array}\right)
\end{split}
\end{equation*}
in probability on an extended probability space as $N \to \infty $, where
$(B^{1},\cdots , B^{n})=(B_{t}^{1},\cdots , B_{t}^{n})_{0\leq t\leq T}$
is an $n$-dimensional Brownian motion independent of 
$W=(W_{t})_{0\leq t\leq T}$.
\end{thm}%%%%%%%%%%%%%%%%%%%%%%%%%%%%%%%%%%%%%%%%%%%%%%%%%%%
\begin{rem}%%%%%%%%%%%%%%%%%%%%%%%%%%%%%%%%%%%%%%%%%%%%%%%%%
Although the Brownian motion $B=(B^{1},\cdots ,B^{n})$ above is not adapted to the filtration
$(\mathcal{F}_{t})_{0 \leq t \leq T}$, the above stochastic integrals make
sense because it is an
$\big( \mathcal{F}_{t} \vee \sigma ( B_{s}: 0 \leq s \leq t ) \big)_{0 \leq t \leq T}$-Brownian motion.
\end{rem}%%%%%%%%%%%%%%%%%%%%%%%%%%%%%%%%%%%%%%%%%%%%%%%%%%%
\begin{proof}%%%%%%%%%%%%%%%%%%%%%%%%%%%%%%%%%%%%%%%%%%%%%%%
By Theorem \ref{ACO2}, we have, 
\begin{equation*}
\begin{split}
&( \Delta t )^{-p/2} \mathrm{Err}_{N}(p) \\
& =
\sum_{m=p+1}^{\infty} \sum_{l=1}^{N}
\frac{ (\Delta t)^{(m-p)/2} }{ \sqrt{m!} }
\mathbf{E} \big[
D_{lT/N}^{m} X^{N} \big\vert \mathcal{G}_{l-1}^{N}
\big]
H_{m} \left( \frac{\Delta W^N_{l}}{ \sqrt{ \Delta t } } \right) .
\end{split}
\end{equation*}
For $m\geq p+2$, by using the integration by parts formula (\ref{DIBP}),
we see that
\begin{equation}
\begin{split}
&
\left\Vert
\sum_{m=p+2}^{\infty} \sum_{l=1}^{N}
\frac{ (\Delta t)^{(m-p)/2} }{ \sqrt{m!} }
\mathbf{E} \big[
D_{lT/N}^{m} X^{N} \big\vert \mathcal{G}_{l-1}^{N}
\big]
H_{m} \left( \frac{\Delta W^N_{l}}{ \sqrt{ \Delta t } } \right)
\right\Vert_{L^{2}}^{2} \\
&=%%%%%%%%
( \Delta t )^{2}
\sum_{k=0}^{\infty} \sum_{l=1}^{N}
\frac{k!}{ (k+p+2)! }
\left\Vert\mathbf{E} \big[
\big( D_{lT/N}^{p+2} X^{N} \big)
H_{k} \left( \frac{\Delta W^N_{l}}{ \sqrt{ \Delta t } } \right)
\big\vert \mathcal{G}_{l-1}^{N}
\big]
%H_{m} \left( \frac{\Delta W^N_{l}}{ \sqrt{ \Delta t } } \right)
\right\Vert_{L^{2}}^{2} \\
&\leq%%%%%
(\Delta t)
\sum_{k=1}^{\infty} \frac{ 1 }{ k^{p+2} }
\times
\sum_{l=1}^{N}
\big\Vert
D_{lT/N}^{p+2} X^{N}
\big\Vert_{L^{2}}^{2} \Delta t \\
& =
(\Delta t)
\sum_{k=1}^{\infty} \frac{ 1 }{ k^{p+2} }
\times
\int_{0}^{T} \hspace{-2mm} \Vert D_{t}^{p+2} X^{N} \Vert_{L^{2}}^{2} dt
\label{INE-COR}%@@@@@@@@@@@@@@@@@@@@@@@@@@@@@@@@@@@@@@@@@@@@ INE-COR
\end{split}
\end{equation}
which goes to zero as $N \to \infty$ for each $p=0,1,\cdots ,n$ by the assumption.

Let us consider the case $m=p+1$. For each $p=0,1,\cdots ,n$, we define a right-continuous
process $L^{p,N}=(L_{t}^{p,N})_{0\leq t\leq T}$
with left-hand side limits by
$$
L_{t}^{p,N} := \sum_{l=1}^{k} H_{p+1} \left( \frac{\Delta W^N_{l}}{ \sqrt{\Delta t} } \right)
\quad \text{if $t_{k-1} \leq t < t_{k}$}
$$
for $k=1,2, \cdots ,N$, and $L_{T}^{p,N}:=L_{t_{N-1}}^{p,N}$.

Since $H_{p+1} \left( \frac{\Delta W^N_{l}}{ \sqrt{\Delta t} } \right)$, $l=1,2,\cdots ,N$
are i.i.d. random variables and
$H_{p+1} \left( \frac{\Delta W^N_{l}}{ \sqrt{\Delta t} } \right)$, $p=0,1, \cdots ,n$
are orthogonal to each other for each $l=1,2,\cdots , N$,
the central limit theorem of finite dimensional distributions of $(\Delta t)^{1/2} L^{p,N}$,
$N=1,2,\cdots$ follows as
for each $0 \leq s < t$, with taking $t_{j-1} \leq s < t_{j}$ and $t_{k-1} \leq t < t_{k}$,
\begin{equation*}
\begin{split}
&
\lim_{N \to \infty} \mathbf{E} \big[
e^{\text{{\scriptsize$\displaystyle
i \sum_{p=0}^{n} \xi_{p} \left\{ (\Delta t)^{1/2} L_{t}^{p,N} - (\Delta t)^{1/2} L_{s}^{p,N} \right\}
$}}
}
\big\vert \mathcal{F}_{s}^{ L^{0,N} } \vee \mathcal{F}_{s}^{ L^{1,N} } \vee \cdots \mathcal{F}_{s}^{ L^{n,N} }
\big] \\
&=%%%%%%%%
\lim_{N \to \infty} \prod_{l=j+1}^{k} \mathbf{E} \big[
e^{
	i
	\text{{\scriptsize$\displaystyle
	\sum_{p=0}^{n} \left( \xi_{p} \sqrt{ t_{k} - t_{j} } \right) \cdot
	(k-j)^{-1/2} H_{p+1} \left( \frac{\Delta W^N_{l}}{\sqrt{\Delta t}} \right)
	$}}
}
\big] \\
&=%%%%%%%%
\lim_{N \to \infty} \prod_{l=j+1}^{k}
\Big\{
1 - \frac{ |\xi |^{2} }{2(k-j)}(t_{k}-t_{j})
+ o \left( \frac{ |\xi |^{2} }{ k-j } \right)
\Big\}
=
e^{ - \frac{ \xi^{2} }{ 2 } (t-s) }.
\end{split}
\end{equation*}
for each $\xi = ( \xi_{0}, \xi_{1}, \cdots , \xi_{n} ) \in \mathbb{R}^{n+1}$,
where $(\mathcal{F}_{t}^{Z})_{0 \leq t \leq T}$ denotes the filtration generated
by a stochastic process $Z=(Z_{t})_{0\leq t \leq T}$ and the little-o-notation 
is with respect to the asymptotics 
when $N \to \infty$(so that $k-j \to \infty$).
This implies that every finite dimensional distribution of $(n+1)$-dimensional process
$( (\Delta t)^{1/2} L^{p,N})_{p=0}^{n}$ converges to that of an $(n+1)$-dimensional
Brownian motion $(B^{0}, B^{1}, \cdots , B^{n})=(B_{t}^{0},B_{t}^{1}, \cdots , B_{t}^{n})_{0\leq t\leq T}$.

Besides, using Kolmogorov's inequality,
we have for each $p=0,1,\cdots ,n$,
\begin{equation*}
\begin{split}
& \lim_{K \to \infty} \limsup_{N \to \infty}
\mathbf{P} \Big(
\sup_{ 0 \leq t \leq T}  \big\vert ( \Delta t )^{1/2} L_{t}^{p,N} \big\vert
\geq K
\Big) \\
& \qquad \leq
\lim_{K\to \infty}
\frac{
	( \Delta t ) \Vert L_{T}^{p,N} \Vert_{L^{2}}^{2}
}{
	K^{2}
}
=0
\end{split}
\end{equation*}
and for each $\varepsilon >0$,
\begin{equation*}
\begin{split}
&
\lim_{\delta \to 0} \limsup_{N \to \infty}
\mathbf{P} \Big(
\inf_{ \substack{ \{ s_{j} \}_{j} \subset [0,T] :\\ | s_{j} -s_{j+1} | > \delta }}
\max_{j} \sup_{t,s \in [s_{j-1},s_{j})} (\Delta t )^{1/2}| L_{t}^{p,N} - L_{s}^{p,N}  |
\geq \varepsilon
\Big) \\
&\leq
\limsup_{N \to \infty}
\mathbf{P} \Big(
\max_{j=1,2,\cdots ,N} \sup_{t,s \in [s_{j-1},s_{j})} (\Delta t )^{1/2}| L_{t}^{p,N} - L_{s}^{p,N}  |
\geq \varepsilon
\Big) \\
&=
\limsup_{N \to \infty}
\mathbf{P} \big( 0 \geq \varepsilon \big) =0.
\end{split}
\end{equation*}
They imply the tightness of $\{ (\Delta t)^{1/2} L^{p,N} \}_{N=1}^{\infty}$(see Billingsley \cite{Bi},
Theorem 13.2). Therefore,
\begin{equation*}
\big\{ \big(
(\Delta t)^{1/2} L^{0,N}, (\Delta t)^{1/2} L^{1,N}, \cdots , (\Delta t)^{1/2} L^{n,N}
\big) \big\}_{N=1}^{\infty}
\end{equation*}
also forms a tight family. Hence we have
$$
(\sqrt{\Delta t} L^{0,N}, \sqrt{\Delta t} L^{1,N}, \cdots , \sqrt{\Delta t} L^{n,N})
\to
(B^{0},B^{1}, \cdots ,B^{n})
$$
in law as $N \to \infty$.
By the Skorohod representation theorem (see Ikeda-Watanabe \cite{IW}, Theorem 2.7\footnote{
On the space of all right-continuous functions with left-hand side limits,
one can endow so-called the {\it Skorohod topology} which is metrizable and makes the space
a complete separable metric space. For details, see Billingsley \cite{Bi}, Chapter 5.
}), we may assume that the above convergence is realized as an almost sure convergence
on an extended probability space. Note that on the probability space we still have $B^{0}=W$ a.s.

Hence we have
\begin{equation*}
\begin{split}
&
\frac{ (\Delta t)^{((p+1)-p)/2} }{ \sqrt{(p+1)!} }
\sum_{l=1}^{N}
\mathbf{E} \big[
D_{lT/N}^{p+1} X^{N} \big\vert \mathcal{G}_{l-1}^{N}
\big]
H_{p+1} \left( \frac{\Delta W^N_{l}}{ \sqrt{ \Delta t } } \right) \\
&=
\frac{ 1 }{ \sqrt{(p+1)!} }
\sum_{l=1}^{N} \mathbf{E} \big[
D_{t_{l}}^{p+1} X^{N} \big\vert \mathcal{G}_{t_{l-1}}
\big]
\Big\{ (\Delta t)^{1/2} L_{t_{l}}^{p,N} - (\Delta t)^{1/2} L_{t_{l-1}}^{p,N} \Big\} \\
%&=
%\frac{ 1 }{ \sqrt{(p+1)!} }
%\int_{0}^{T} \hspace{-2mm}
%\mathbf{E} \big[
%D_{t}^{p+1} X^{N} \big\vert \mathcal{G}_{t}
%\big]
%dB_{t}^{p}
%\quad \text{in law} \\
&\to
\frac{ 1 }{ \sqrt{(p+1)!} }
\int_{0}^{T} \hspace{-2mm}
\mathbf{E} \big[
D_{t}^{p+1} X \big\vert \mathcal{G}_{t}
\big]
dB_{t}^{p}
\quad\text{in probability as $ N \to \infty $}
\end{split}
\end{equation*}
simultaneously for $p=0,1,\cdots ,n$.
\end{proof}%%%%%%%%%%%%%%%%%%%%%%%%%%%%%%%%%%%%%%%%%%%%%%%%%

Substituting $p=0$ into the inequality (\ref{INE-COR}) in the proof of Theorem \ref{WEAK},
we also obtain the following
\begin{coro}\label{1stEE}%%%%%%%%%%%%%%%%%%%%%%%%%%%%%%%%%%%%%%%%%%%%%%%%
If
$\displaystyle
\sup_{N}
\int_{0}^{T} \hspace{-2mm}
\Vert D_{t}^{2} X^{N} \Vert_{L^{2}}^{2}
dt
<\infty
$
then we have
$$
\big\Vert
X^{N}
-
\Big\{
	\mathbf{E} [X^{N}]
	+
	\mathbf{E} [ D_{lT/N} X^{N} \vert \mathcal{G}_{l-1}^{N} ] \Delta W^N_{l}
\Big\}
\big\Vert_{L^{2}}
=
O(N^{-1/2})
$$
as $N\to \infty$.
\end{coro}%%%%%%%%%%%%%%%%%%%%%%%%%%%%%%%%%%%%%%%%%%%%%%%%%%

%%%%%%%%%%%%%%%%%%%%%%%%%%%%%%%%%%%%%%%%%%%%%%%%%%%%%%%%%%%%
%\input{old-ASYMP.tex}*** OLD VERSION of the subsection ASYMP
%\input{old-ASYMP2.tex}*** OLD VERSION of the subsection 2nd version of ASYMP
%%%%%%%%%%%%%%%%%%%%%%%%%%%%%%%%%%%%%%%%%%%%%%%%%%%%%%%%%%%%
%%%%%%%%%%%%%%%%%%%%%%%%%%%%%%%%%%%%%%%%%%%%%%%%%%%%%%%%%%%%

\subsection{The Cases with ``Finite Dimensional" Functionals}\label{ASYMP}
%%%%%%%%%%%%%%%%%%%%%%%%%%%%%%%%%%%%%%%%%%%%%%%%%%%%%%%%%%%%
We have seen that the martingale representation error is of an  
order $ 1/2 $ for a smooth functional. In this section, 
we will observe that 
for a non-smooth functional, the order is related to  
its fractional differentiability if it behaves eventually like a
finite dimensional functional. This parallels with 
the corresponding results in the cases of the tracking error
as we have pointed out in Introduction. 

\if Suppose that $ X $ is a finite dimensional functional; namely,  
$ X = \mathbf{E} [X | \mathcal{G}^m ] $ for some $ m $.
Then we can identify $ X $ with some $ F \in L^2 (\mathbb{R}^m, \mu^m ) $. 
We note that 
the error estimation of such a finite dimensional functional reduces 
to the one dimensional case. In fact, for $ N > m $ we have
\begin{equation*}
\begin{split}
&X - \mathbf{E} [X] 
- \sum_{s} \mathbf{E} [\partial_s X | \mathcal{G}^N_s] 
\Delta W^N_s \\
& \qquad =
\sum_{l=1}^m \big( X_l - X_{l-1}  
- \sum_{\frac{(l-1)}{m} \leq s < \frac{l}{m} } 
\mathbf{E} [\partial_s X_l | \mathcal{G}^N_s] \Delta W^N_s \big) \\
& \qquad =: \sum_{l=1}^m ( \mbox{$ l $-th Mart.Error}),
\end{split}
\end{equation*}
where $ X_l :=  \mathbf{E} [ X | \mathcal{G}^N_{l/m}] 
=  \mathbf{E} [ X | \mathcal{G}^m_{l/m}] $, and notice that 
the $ L^2 $-norm of each ``$ l $-th Mart.Error" is the error with respect to 
a one dimensional functional of $ \Delta W^m_{l/m} $. 
Here we denote $ \Delta W^{M}_s = W_{s+M^{-1}} - W_s $
with $ s = k/M $ for some positive integers $ M $ and $ k \leq M $.  
\fi

Let us start with one-dimensional cases. 
Let $ F \in L^2 (\mathbb{R}, \mu_T) $, where $ \mu_T $ is the Gaussian measure 
with variance $ T > 0 $.  
%Here we rearrange $ N \Delta t = T $ and 
%study asymptotics when $ N \to \infty $.  
Then, since 
\begin{equation*}
\frac{\partial^k}{\partial x_l^k} F (x_1 + \cdots + x_N ) 
= F^{(k)} (x_1 + \cdots + x_N ), 
\end{equation*}
we have, for $ k_1 + \cdots +k_N = n $, 
\begin{equation*}
\begin{split}
&
\mathbf{E} \big[
D_{t_1^{(N)}}^{k_{1}} \cdots D_{t_l^{(N)}}^{k_{l}} F (W_T)
\big]^{2}
=
\mathbf{E} \big[ F^{(n)}(W_{T}) \big]^{2} \\
&=
\frac{n!}{T^{n}}
\mathbf{E} \big[ F(W_{T}) H_{n}\left( \frac{W_{T}}{\sqrt{T}} \right) \big]^{2}
=
\frac{n!}{T^{n}} \Vert J_{n}F (W_T) \Vert_{L^{2}}^{2},
\end{split}
\end{equation*}
irrespective of $ l $ and $ N $. 
Here $ J_n $ is the projection to the $ n $-th chaos. 
With this observation in mind, we understand 
the following property as a finite-dimensionality of a sequence;
let $ \{ F^N \} $ be such that each $ F^N $ being $ \mathcal{G}^N $-measurable
and that
\begin{equation}\label{fdp}
\begin{split}
& \sup_{ k_{1}+\cdots + k_{N} = n }
(E [D_{t_1^{(N)}}^{k_1} \cdots D^{k_N}_{t_N^{(N)}}F])^2
=
O \left( \frac{ n! \Vert J_{n}F^{N} \Vert^{2} }{ T^{n} } \right) \\
& \hspace{2cm} \text{uniformly in $n=2,3,\cdots$ as $N \to \infty$.}
\end{split}
\end{equation}
Note that a sequence composed of a one dimensional functional $ F (W_T) $
satisfies the above property trivially.
Furthermore, the multi dimensional case where 
$ F^N \equiv F, N=1,2,\cdots $ for 
some $ F \in L^2 (\mathcal{G}^m_m) $, $ 2 \leq m <\infty $
satisfies (\ref{fdp}) as well. 
In fact, for arbitrary non-negative integers $ k_1, \cdots, k_N $
with $ k_1 + \cdots + k_N =n $, 
the relation  
\begin{equation*}
(E [D_{t_1^{(N)}}^{k_1} \cdots D^{k_N}_{t_N^{(N)}}F])^2
= (E[D_{t_1^{(m)}}^{l_1} \cdots D_{t_m^{(m)}}^{l_m} F] )^2,
\end{equation*}
where $ l_j = k_{N'(j-1)+1} + \cdots +  k_{N'(j-1)+N'} $,  
implies 
\begin{equation*}
\begin{split}
& (E [D_{t_1^{(N)}}^{k_1} \cdots D^{k_N}_{t_N^{(N)}}F])^2 
 \leq \frac{l_1 !\cdots l_m !}{T^n}  m^n  \Vert J_n F \Vert^2 \\
& = \frac{l_1 !\cdots l_m !}{T^n}
\sum_{l'_1+\cdots+l'_m=n} \frac{n!}{l'_1 ! \cdots l'_m !} \Vert J_n F \Vert^2 \\
&\leq \frac{n!}{T^n} \Vert  J_n F \Vert^2. 
\end{split}
\end{equation*}

\begin{thm}\label{ErrEst}%%%%%%%%%%%%%%%%%%%%%%%%%%%%%%%%%%%
Suppose that we are 
given a sequence of $F^{N} \in \mathbb{D}_{2,-\infty}^{(N)}$, $N=1,2,\cdots$
satisfying 
$$
\sup_{N} \Vert F^{N} \Vert_{\mathbb{D}_{2,s}}^{2} < \infty
$$
for some $0 \leq s \leq 1$ and 
the ``finite-dimensional property" (\ref{fdp}). 
Then
$$
\big\Vert \mbox{{\rm 1-Mart.Err}}(F^{N}) \big\Vert_{L^{2}}^{2}
=
O(N^{-s/2})
\quad \text{as $N \to \infty$.}
$$
\end{thm}%%%%%%%%%%%%%%%%%%%%%%%%%%%%%%%%%%%%%%%%%%%%%%%%%%%
\if
The proof is based on the following lemma, which is  
fundamental for the estimation of the martingale error
since it involves only equalities. 
\begin{lemma}\label{Er-Norm}%%%%%%%%%%%%%%%%%%%%%%%%%%%%%%
For every $F \in \mathbb{D}_{2,-\infty}^{(N)}$, we have
$$
\Vert \mbox{{\rm 1-Mart.Err}}(F) \Vert_{L^{2}}^{2}
=
\sum_{k=2}^{\infty} \frac{ (\Delta t)^{ k } }{ k! }
\Big\{
I_{N}^{k} ( F )
-
k \sum_{l=1}^{N}
I_{l-1}^{ k-1 } ( \partial_{l} F )
\Big\}
$$
where
$\displaystyle
I_{l}^{n}(F)
:=
\sum_{ k_{1} + \cdots + k_{l} = n }^{\infty}
\frac{ n! }{ k_{1}! \cdots k_{l-1}! k_{l}! }
\mathbf{E} \big[
\partial_{1}^{k_{1}} \cdots \partial_{l}^{k_{l}} F
\big]^{2}.
$
\end{lemma}%%%%%%%%%%%%%%%%%%%%%%%%%%%%%%%%%%%%%%%%%%%%%%%%%%
\begin{proof}%%%%%%%%%%%%%%%%%%%%%%%%%%%%%%%%%%%%%%%%%%%%%%%
By Theorem \ref{ACO2}, we see that
\begin{equation*}
\begin{split}
&
\Vert \mbox{{\rm 1-Mart.Err}}(F) \Vert_{L^{2}}^{2} \\
&=
\sum_{l=1}^{N} \sum_{k_{l}=2}^{\infty}
\sum_{ k_{1}, \cdots , k_{l-1} = 0 }^{\infty}
\frac{ (\Delta t)^{ k_{1} + \cdots + k_{l-1} + k_{l}} }{ k_{1}! \cdots k_{l-1}! k_{l}! }
\mathbf{E} \big[
\partial_{1}^{k_{1}} \cdots \partial_{l-1}^{k_{l-1}} \partial_{l}^{k_{l}} F
\big]^{2} \\
&=
\sum_{l=1}^{N} \sum_{k_{l}=2}^{\infty}
\sum_{n=0}^{\infty}
\frac{ (\Delta t)^{ n+k_{l} } }{ n! k_{l}! }
I_{l-1}^{n} \big( \partial_{l}^{k_{l}} F \big) \\
&=
\sum_{l=1}^{N} \sum_{m=2}^{\infty} \frac{ (\Delta t)^{ m } }{ m! }
\sum_{k_{l}=2}^{m}
\binom{ m }{ k_{l} }
I_{l-1}^{ m-k_{l} } \big( \partial_{l}^{k_{l}} F \big) .
\end{split}
\end{equation*}
One can check easily the following binomial theorem-like formula:
$$
I_{l}^{n}(F)
=
\sum_{k=0}^{n} \binom{n}{k} I_{l-1}^{n-k} \big( \partial_{l}^{k}F \big) .
$$
Using this, we see that
\begin{equation*}
\begin{split}
&
\Vert \mbox{{\rm 1-Mart.Err}}(F) \Vert_{L^{2}}^{2} \\
&=
\sum_{m=2}^{\infty} \frac{ (\Delta t)^{ m } }{ m! }
\sum_{l=1}^{N}
\Big\{
I_{l}^{m} \big( F \big)
-
I_{l-1}^{m} \big( F \big)
-
m I_{l-1}^{ m-1 } \big( \partial_{l} F \big)
\Big\} \\
&=
\sum_{m=2}^{\infty} \frac{ (\Delta t)^{ m } }{ m! }
\Big\{
I_{N}^{m} \big( F \big)
-
m \sum_{l=1}^{N}
I_{l-1}^{ m-1 } \big( \partial_{l} F \big)
\Big\} .
\end{split}
\end{equation*}

\end{proof}%%%%%%%%%%%%%%%%%%%%%%%%%%%%%%%%%%%%%%%%%%%%%%%%%
\fi

\begin{proof}%[Proof of Theorem \ref{ErrEst}]%%%%%%%%%%%%%%%%%%%%%%%%%%%%%%%%%%%%%%%%%%%%%%%
By observing (\ref{PCO2}), we notice that 
\begin{equation*}
\begin{split}
& \Vert \mbox{{\rm 1-Mart.Err}}(F^{N}) \Vert_{L^{2}}^{2} \\
&=
\sum_{l=1}^{N}
\sum_{\substack{k_{1} + \cdots + k_{l}=n \\ k_{l} \geq 2}}
\frac{ n! }{ k_{1}! \cdots k_{l}! } (\Delta t)^{n}
\mathbf{E} \big[
\partial_{1}^{k_{1}} \cdots \partial_{l}^{k_{l}} F^{N}
\big]^{2}
\end{split}
\end{equation*}
for each $n=2,3,\cdots$.
By the assumption, there is a constant $C>0$ such that
$$
\sup_{ k_{1}+\cdots + k_{l} = n }
\mathbf{E} \big[ \partial_{1}^{k_{1}} \cdots \partial_{l}^{k_{l}} F^{N} \big]^{2}
\leq
C \frac{ n! \Vert J_{n}F^{N} \Vert^{2} }{ T^{n} }
$$
for each $n=2,3,\cdots$ and $N=1,2,\cdots$ and 
the multinomial theorem yields that
\begin{equation*}
\begin{split}
&\sum_{ \substack{k_{1} + k_{2} + \cdots + k_{l} = n \\ k_{l} \geq 2} }
\frac{ n! }{ k_{1}! \cdots k_{l}! } ( \Delta t )^{n} \\
&=
\left( \frac{lT}{N} \right)^{n}
-
\left( \frac{ (l-1)T }{ N } \right)^{n}
-
n \frac{T}{N} \left( \frac{ (l-1)T }{ N } \right)^{n-1}.
\end{split}
\end{equation*}
Putting them together, we have
\begin{equation}\label{ineq1}
\begin{split}
&
\big\Vert \mbox{{\rm 1-Mart.Err}}(F^{N}) \big\Vert_{L^{2}}^{2} \\
&\leq
C
\sum_{n=2}^{\infty}
\Big\{
1 - n \frac{1}{N} \sum_{l=0}^{N-1} \left( \frac{ l }{ N } \right)^{n-1}
\Big\}
\Vert J_{n} F^{N} \Vert_{L^{2}}^{2} \\
&=
CN^{-s}
\sum_{n=2}^{\infty} \frac{ N^{s} }{ n^{s-1} }
\Big\{
\frac{1}{n} - \frac{1}{N} \sum_{l=0}^{N-1} \left( \frac{ l }{ N } \right)^{n-1}
\Big\}
n^{s} \Vert J_{n} F^{N} \Vert_{L^{2}}^{2}
\end{split}
\end{equation}
for each $s \in \mathbb{R}$. 

On the other hand, since we have 
\begin{equation*}
\begin{split}
&I_{n,N}
:=
\frac{1}{n} - \frac{1}{N} 
\sum_{l=0}^{N-1} \left( \frac{ l }{ N } \right)^{n-1} \\
&=
\sum_{l=0}^{N-1} \int_{ l/N }^{ (l+1)/N }
\Big\{ x^{n-1} - \Big( \frac{l}{N} \Big)^{n-1} \Big\}
dx
> 0,
\end{split}
\end{equation*}
$ I_{n,N} \leq 1/n  $, and 
\begin{equation*}
\begin{split}
I_{n,N}
\leq
\frac{1}{N} \sum_{l=0}^{N-1}
\Big\{
\left( \frac{l+1}{N} \right)^{n-1} - \left( \frac{l}{N} \right)^{n-1}
\Big\}
=
\frac{1}{N},
\end{split}
\end{equation*}
we notice that
\begin{equation}\label{ineq2}
I_{n,N}
=
I_{n,N}^{s} I_{n,N}^{1-s}
\leq
\left( \frac{1}{N} \right)^{s} \left( \frac{1}{k} \right)^{1-s}
\end{equation}
for every $0\leq s \leq 1$.

By (\ref{ineq1}) and (\ref{ineq2}), we finally have  
\begin{equation*}
\begin{split}
&
\big\Vert \mbox{{\rm 1-Mart.Err}}(F^{N}) \big\Vert_{L^{2}}^{2} \\
&\hspace{5mm}\leq
CN^{-s} \sum_{n=2}^{\infty} n^{s} \Vert J_{n}F^{N} \Vert_{L^{2}}^{2}
\leq
CN^{-s} \sup_{N} \Vert F^{N} \Vert_{\mathbb{D}_{2,s}^{(N)}}^{2}.
\end{split}
\end{equation*}
\end{proof}%%%%%%%%%%%%%%%%%%%%%%%%%%%%%%%%%%%%%%%%%%%%%%%%%

\if0%%%%%%%%%%%%%%%%%%%%%%%%%%%%%%%%%%%%%%%%%%%%%%%%%%%%%%%%%%%%%%%%%% COMMENTED OUT BEGIN
Bearing in mind the left continuity of the mapping
$s^{\prime} \mapsto \Vert F \Vert_{ \mathbb{D}_{2,s^{\prime}}^{(1)} }$,
we have the following 
%we can improve Corollary \ref{Sob-Ord} slightly as follows.
\begin{coro}%%%%%%%%%%%%%%%%%%%%%%%%%%%%%%%%%%%%%%%%%%%%%%%%%
Assume that $F\in \mathbb{D}_{2,s}^{(1)}$ for some $0 \leq s \leq 1$.
Then there exists $\varepsilon >0$ such that
$$
\Vert
\mathrm{1\text{-}Mart.Err}
\Vert_{L^{2}}
=
O(N^{- \min \{ s+\varepsilon ,  1 \} /2})
$$
as $N \to \infty$.
\end{coro}%%%%%%%%%%%%%%%%%%%%%%%%%%%%%%%%%%%%%%%%%%%%%%%%%%

\begin{proof}
The left continuity of
$s^{\prime} \mapsto \Vert F \Vert_{\mathbb{D}_{2,s^{\prime}}^{(1)}}$
implies that there exists $\varepsilon >0$ such that $F \in \mathbb{D}_{2,s+\varepsilon}^{(1)}$
and then by Theorem \ref{ErrEst}
can be applied to guarantee the convergence rate of
$O(N^{ -\min \{ s+\varepsilon , 1 \} /2})$.
\end{proof}
\fi%%%%%%%%%%%%%%%%%%%%%%%%%%%%%%%%%%%%%%%%%%%%%%%%%%%%%%%%%%%%%%%%%%% COMMENTED OUT END

\subsection{A Study on Additive Functionals}\label{ADDsec}
In this subsection, we study sequences of
``additive functionals", 
\begin{equation}\label{Adf}
\begin{split}
& F^{N} := \sum_{i=1}^{N} f_N (t_i, W_{t^{(N)}_{i}}) \Delta t \\
& \text{where $ f_N (t_i, \cdot) $, $ i=1,\cdots N $
is a sequence in $ \mathbb{D}^{(1)}_{2,-\infty} $.}
\end{split}
\end{equation}
We are interested in the conditions 
for the sequence to be ``finite-dimensional" in the sense of 
(\ref{fdp}). 

We define an index to control the finite-dimensionality. 
Let 
\begin{equation*}
A_l := (\sum_{i=l}^N 
i^{-n/2} E[ f_N (t_i,W_{t_i})H_n (W_{t_i}/\sqrt{t_i})])^2
\end{equation*}
and
\begin{equation*}
\alpha_{N,n}(F^N) := \begin{cases}
0 &  \sum_{l=1}^N A_l\{l^n - (l-1)^n\} = 0 \\
\frac{N^n \sup A_l }{\sum_{l=1}^N A_l\{l^n - (l-1)^n\}} 
& \text{otherwise}.
\end{cases} 
\end{equation*}
Then, we have the following criterion. 
\begin{prop}%%%%%%%%%%%%%%%%%%%%%%%%%%%%%%%%%%%%%%%%%%%%%%%%
The sequence $ \{ F_N \} $ of (\ref{Adf}) satisfies (\ref{fdp})
if and only if 
\begin{equation*}
\sup_n \sup_N \alpha_{n,N} (F_N) < \infty. 
\end{equation*}
\end{prop}%%%%%%%%%%%%%%%%%%%%%%%%%%%%%%%%%%%%%%%%%%%%%%%%%%
\begin{proof}
For arbitrary non-negative integers $ k_1, \cdots, k_N $
with $ k_1 + \cdots + k_N =n $, we have
\begin{equation*}
\begin{split}
& E [D_{t_1}^{k_1} \cdots D^{k_N}_{t_N}F^N]
= \sum_{i=1}^N 1_{\{k_{i+1}=\cdots=k_N=0\}} E[ f_N^{(n)} (t_i,W_{t_i})] \Delta t \\
& = (n!)^{1/2} (\Delta t)^{(2-n)/2}
\sum_{i=1}^N 1_{\{k_{i+1}=\cdots=k_N=0\}} i^{-n/2} E[ f_N (t_i,W_{t_i})H_n (W_{t_i}/\sqrt{t_i})] .
\end{split}
\end{equation*}
If further $ k_l \geq 1 $ and $ k_{l+1} =\cdots k_N =0 $ for some $ l $,
then
\begin{equation*}
\begin{split}
&E [D_{t_1}^{k_1} \cdots D^{k_l}_{t_l}F^N] \\
&= (n!)^{1/2} (\Delta t)^{(2-n)/2} \sum_{i=l}^N 
i^{-n/2} E[ f_N (t_i,W_{t_i})H_n (W_{t_i}/\sqrt{t_i})] \\
&=(n!)^{1/2} (\Delta t)^{(2-n)/2} A_l^{1/2} . 
\end{split}
\end{equation*}
Therefore, 
\begin{equation}\label{supDD}
\sup_{k_1+\cdots+k_N=n}
(E [D_{t_1}^{k_1} \cdots D^{k_N}_{t_N}F^N])^2
= n! (\Delta t)^{(2-n)} \sup_{l=1, \cdots, N} A_l
\end{equation}

On the other hand, 
we have
\begin{equation}\label{JN}
\begin{split}
& \Vert J_n F^N \Vert^2 \\
&= \sum_{l=1}^N \sum_{\substack{k_1+\cdots+k_l=n \\ k_l \geq 1}}
(E[F^N H_{k_1} ( \Delta W_{1} / \sqrt{\Delta t} ) 
\cdots H_{k_l}( \Delta W_{l} / \sqrt{\Delta t} )])^2 \\
&=  \sum_{l=1}^N \sum_{\substack{k_1+\cdots+k_l=n \\ k_l \geq 1}}
\frac{(\Delta t)^n}{k_1! \cdots k_l!} 
(E [D_{t_1}^{k_1} \cdots D^{k_l}_{t_l}F^N])^2 \\ 
&=  \sum_{l=1}^N A_l \sum_{\substack{k_1+\cdots+k_l=n \\ k_l \geq 1}}
\frac{(\Delta t)^2 n!}{k_1! \cdots k_l!} 
= (\Delta t )^2 \sum_{l=1}^N A_l \{ l^n -(l-1)^n\}. 
\end{split}
\end{equation}
Putting (\ref{supDD}) and (\ref{JN}) together, 
we have
\begin{equation*}
\begin{split}
\sup_{k_1+\cdots+k_N=n}\frac{
(E [D_{t_1}^{k_1} \cdots D^{k_N}_{t_N}F^N])^2}
{\Vert J_n F^N \Vert^2}
&= \frac{n!}{T^n}\frac{N^n \sup A_l }{\sum_{l=1}^N A_l\{l^n - (l-1)^n\}} \\
&=\frac{n!}{T^n}\alpha_{N,n}(F^N).
\end{split}
\end{equation*}
Note that $ \Vert J_n F^N \Vert^2 = 0 $ implies both $ \alpha_{N,n}(F^N) = 0 $ 
and 
\begin{equation*}
\sup_{k_1+\cdots+k_N=n} E [D_{t_1}^{k_1} \cdots D^{k_N}_{t_N}F^N])^2 = 0.
\end{equation*}
\end{proof}

\begin{coro}
If
\begin{equation*}
\sup_N \frac{\sup_l A_l}{\inf_l A_l} < \infty,
\end{equation*}
then $ \{F^N\} $ is finite-dimensional. 
\end{coro}
\begin{proof}
Since
\begin{equation*}
\sum_{l=1}^N A_l\{l^n - (l-1)^n\} \geq 
\inf_l A_l \sum_{l=1}^N \{l^n - (l-1)^n\} = N^n \inf_l A_l,
\end{equation*}
we see that 
\begin{equation*}
\alpha_{n,N}(F^N) \leq \frac{\sup_l A_l}{\inf_l A_l}.
\end{equation*}
\end{proof}

%%%%%%%%%%%%%%%%%%%%%%%%%%%%%%%%%%%%%%%%%%%%%%%%%%%%%%%%%%%%
\subsection{Asymptotic Analysis of the Martingale Representation Error 
of a Discretization 
of Brownian Occupation Time}
%%%%%%%%%%%%%%%%%%%%%%%%%%%%%%%%%%%%%%%%%%%%%%%%%%%%%%%%%%%%
The sequence of Riemann sum approximations
\begin{equation}\label{ocp}
F^{N} := \sum_{i=1}^{N} 1_{[0,\infty )} (W_{t_{i}}) \Delta t, 
\quad N \in \mathbb{N} 
\end{equation}
of the Brownian occupation time $ \int_0^T 1_{[0,\infty)} (W_s)\,ds $ is 
an interesting example where an explicit calculation is possible. 
We first prove that the sequence is not finite-dimensional 
in the sense of (\ref{fdp}). 
However, it is rather difficult to check if the
condition for Corollary \ref{1stEE} is satisfied. 
Instead, 
by a direct calculation the martingale representation error 
of the sequence is proven to be of order $ 1/2 $. 

\begin{prop}
The index $ \alpha_{n,N} (F^N) $ of the sequence (\ref{ocp})
is not bounded. 
\end{prop}
\begin{proof}
First, we observe that   
\begin{equation*}
\begin{split}
A_l & =
\left( \sum_{i=l}^N i^{-n/2} \mathbf{E} [1_{[0,\infty )} (W_{t_{i}}) 
H_n(W_{t_i}/\sqrt{t_i}) ] \right)^2\\
&= \left( \sum_{i=l}^N i^{-n/2}t^{1/2}_i n^{-1/2} 
\mathbf{E} [\delta_0 (W_{t_{i}}) 
H_{n-1} (W_{t_i}/\sqrt{t_i}) ] \right)^2 \\
&=
(2\pi n)^{-1} \big( H_{n-1} (0) \big)^2 
\left( \sum_{i=l}^N i^{-n/2} \right)^2. 
\end{split}
\end{equation*}
Then, we now see that 
\begin{equation}\label{index}
\alpha_{n,N} (F^N) 
=
\frac{N^n \left( \sum_{i=1}^N i^{-n/2} \right)^2}{
\sum_{l=1}^N \left( \sum_{i=l}^N i^{-n/2} \right)^2
\{ l^n - (l-1)^n\}} .
\end{equation} 
First, we estimate the numerator of (\ref{index}).
We let $ n \geq 5 $. Then
\begin{equation}\label{num}
\begin{split}
& N^n \left( \sum_{i=1}^N i^{-n/2} \right)^2
= N^2 \left( \sum_{i=1}^N 
\left(\frac{i}{N}\right)^{-n/2} \frac{1}{N} \right)^2 \\
& \geq N^2 \left(\int_{1/N}^1 x^{-n/2} dx \right)^2 
= N^2 \left\{ \frac{2}{n-2} (N^{(n-2)/2}-1) \right\}^2. 
\end{split}
\end{equation}
Next, the denominator is estimated as follows:
\begin{equation*}
\begin{split}
&\sum_{l=1}^N \left( \sum_{i=l}^N i^{-n/2} \right)^2
\{ l^n - (l-1)^n\} \\
&= N^2 \sum_{l=1}^N \left( \sum_{i=l}^N \left(
\frac{i}{N}\right)^{-n/2} \frac{1}{N} \right)^2 
\left\{ \left(\frac{l}{N}\right)^n - \left(\frac{l-1}{N}\right)^n
\right\} \\
&\leq  N^2 \sum_{l=1}^N \left( \int_{l/N}^1 x^{-n/2} dx + 
\left(\frac{l}{N}\right)^{-n/2} \frac{1}{N} \right)^2 \\
& \hspace{3cm} \times \left\{ \left(\frac{l}{N}\right)^n - \left(\frac{l-1}{N}\right)^n
\right\} \\
&\leq N^2 \sum_{l=1}^N \left( \frac{2}{n-2} 
\left\{ \left(\frac{l}{N}\right)^{(2-n)/2} - 1 \right\} + 
\left(\frac{l}{N}\right)^{-n/2} \frac{1}{N} \right)^2 \\
& \hspace{3cm} 
\times \left\{ \left(\frac{l}{N}\right)^n - \left(\frac{l-1}{N}\right)^n
\right\} \\
&=
N^{2} \big\{ J_{1}^{N} + J_{2}^{N} + J_{3}^{N}  \big\}
\end{split}
\end{equation*}
where
\begin{equation*}
\begin{split}
J_{1}^{N}
&:=
(n-2)^{-2} \sum_{l=1}^{N}
\Big\{ \Big( \frac{l}{N} \Big)^{(2-n)/2} - 1 \Big\}^{2}
\Big\{ \Big( \frac{l}{N} \Big)^{n} - \Big( \frac{l-1}{N} \Big)^{n} \Big\}, \\
%%%%%%%%%%
J_{2}^{N}
&:=
\frac{2 (n-2)^{-1}}{N} \sum_{l=1}^{N}
\Big\{ \Big( \frac{l}{N} \Big)^{1-n} - \Big( \frac{l}{N} \Big)^{-n/2} \Big\}
\Big\{ \Big( \frac{l}{N} \Big)^{n} - \Big( \frac{l-1}{N} \Big)^{n} \Big\}
\end{split}
\end{equation*}
and
\begin{equation*}
J_{3}^{N}
:=
\frac{1}{N^{2}}
\sum_{l=1}^{N} \Big( \frac{l}{N} \Big)^{-n}
\Big\{ \Big( \frac{l}{N} \Big)^{n} - \Big( \frac{l-1}{N} \Big)^{n} \Big\}.
\end{equation*}
It is easy to see that
$\sup_{N} J_{2}^{N} < \infty$
and
$\lim_{N\to\infty} J_{3}^{N}=0$.
Since $J_{1}^{N}$ behaves like
$$
(n-2)^{-2}
\int_{0}^{1}
\big\{ x^{(2-n)/2} - 1 \big\}^{2} n x^{n-1} dx < \infty
$$
as $N\to\infty$, it is also seen that $\sup_{N} J_{1}^{N} < \infty$.
Therefore, there is a constant $ C_n $ independent of $ N $ but possibly 
dependent on $ n $ such that
\begin{equation}\label{deno}
\sum_{l=1}^N \left( \sum_{i=l}^N i^{-n/2} \right)^2
\{ l^n - (l-1)^n\}
\leq
N^{2} C_{n}.
\end{equation}

From (\ref{num}) and (\ref{deno}), we see that 
$ \sup_N \alpha_{n,N} = \infty $. 
\end{proof}

Our main result in this subsection is the following. 
\begin{thm}\label{Occ}%%%%%%%%%%%%%%%%%%%%%%%%%%%%%%%%%%%%%%
\if Let
$\displaystyle
F^{N} := \sum_{i=1}^{N} 1_{[0,\infty )} (W_{t_{i-1}}) \Delta t
$,
then as $N\to\infty$ \fi
It holds that 
\begin{equation*}
\begin{split}
\Vert \text{{\rm \mbox{1-Mart.Err}}}(F^{N}) \Vert_{L^{2}}
=
O(N^{-1/2}). 
\end{split}
\end{equation*}
\end{thm}%%%%%%%%%%%%%%%%%%%%%%%%%%%%%%%%%%%%%%%%%%%%%%%%%%%

\begin{proof}%%%%%%%%%%%%%%%%%%%%%%%%%%%%%%%%%%%%%%%%%%%%%%%
By Theorem \ref{ACO2}, we have
\begin{equation}\label{Basic1}
\begin{split}
&
\Vert \mathrm{Err}_{N} \Vert_{L^{2}}^{2} \\
&=
\sum_{l=1}^{N} \sum_{k=2}^{\infty}
\mathbf{E} \Big[
\mathbf{E} \big[
\sum_{i=1}^{N} 1_{[0,\infty )} (W_{t_{i}}) \Delta t
H_{k} \left( \frac{\Delta W^N_{l}}{\sqrt{\Delta t}} \right)
\big\vert \mathcal{G}_{l-1}^{N}
\big]^{2}
\Big] \\
&=
\sum_{l=1}^{N} \sum_{k=2}^{\infty}
\frac{ (\Delta t)^{k} }{ k! }
\mathbf{E} \Big[
\mathbf{E} \big[
\sum_{i=l}^{N} 1_{[0,\infty )}^{(k)} (W_{t_{i}}) \Delta t
\big\vert \mathcal{G}_{l-1}^{N}
\big]^{2}
\Big]. \\
\end{split}
\end{equation}
For $ l \geq 2 $, 
by the Hermite expansion in $ L^2 (\mathbf{R}, \mu_{t_{l-1}}) $,  
\begin{equation*}
\begin{split}
&
\mathbf{E} \big[
\sum_{i=l}^{N} 1_{[0,\infty )}^{(k)} (W_{t_{i}}) \Delta t
\big\vert \mathcal{G}_{l-1}^{N}
\big] \\
&=
\sum_{n=0}^{\infty}
\frac{ ( t_{l-1} )^{n/2} }{ \sqrt{n!} }
\mathbf{E} \big[
\sum_{i=l}^{N} 1_{[0,\infty )}^{(n+k)} (W_{t_{i}}) \Delta t
\big]
H_{n} \left( \frac{ W_{t_{l-1}} }{ \sqrt{t_{l-1}} } \right),
\end{split}
\end{equation*}
and by Parseval's identity we have
\begin{equation}\label{Parseval}
\begin{split}
&
E\left[ \mathbf{E} \big[
\sum_{i=l}^{N} 1_{[0,\infty )}^{(k)} (W_{t_{i}}) \Delta t
\big\vert \mathcal{G}_{l-1}^{N}
\big]^2 \right] \\
&=
\sum_{n=0}^{\infty}
\frac{ ( t_{l-1} )^{n} }{ n! }
\mathbf{E} \big[
\sum_{i=l}^{N} 1_{[0,\infty )}^{(n+k)} (W_{t_{i}}) \Delta t
\big]^2.
\end{split}
\end{equation}
Note that (\ref{Parseval}) is also valid for $ l=1 $ 
with the conventions $ t_0=0 $ and $ t_0^0 =1 $. 
Plugging (\ref{Parseval}) into (\ref{Basic1}), we have
\begin{equation*}
\begin{split}
&
\Vert \mathrm{Err}_{N} \Vert_{L^{2}}^{2} \\
&=
\sum_{l=1}^{N} \sum_{k=2}^{\infty}\sum_{n=0}^{\infty}
\frac{ (\Delta t)^{k} }{ k! }
\frac{ ( t_{l-1} )^{n} }{ n! }
\mathbf{E} \big[
\sum_{i=l}^{N} 1_{[0,\infty )}^{(n+k)} (W_{t_{i}}) \Delta t
\big]^2.
\end{split}
\end{equation*}
By the renumbering $ (n+k,n) \mapsto (k,n) $, 
we have 
\begin{equation*}
\begin{split}
&
\Vert \mathrm{Err}_{N} \Vert_{L^{2}}^{2} \\
&=
\sum_{l=1}^{N} \sum_{k=2}^{\infty}\frac{1}{k!}
\mathbf{E} \big[
\sum_{i=l}^{N} 1_{[0,\infty )}^{(k)} (W_{t_{i}}) \Delta t
\big]^2
\sum_{n=0}^{k-2}
\frac{ k!  }{ (k-n)!n!}(\Delta t)^{k}( t_{l-1} )^{n},
\end{split}
\end{equation*}
by keeping the conventions on $ t_0 $.  
With a use of the binomial theorem,  
\begin{equation*}
\begin{split}
&
\Vert \mathrm{Err}_{N} \Vert_{L^{2}}^{2} \\
&=
\sum_{l=1}^{N} \sum_{k=2}^{\infty}
\frac{1}{k!}
\mathbf{E} \big[
\sum_{i=l}^{N} 1_{[0,\infty )}^{(k)} (W_{t_{i}}) \Delta t
\big]^{2} \\
& \hspace{2cm} \times
\Big\{
( t_{l} )^{k} - ( t_{l-1} )^{k} - k ( \Delta t )( t_{l-1} )^{k-1}
\Big\} .
\end{split}
\end{equation*}

Then, on one hand, for $ l \geq 1 $ and $ k \geq 2 $, 
\begin{equation*}
\begin{split}
&
\mathbf{E} \big[
\sum_{i=l}^{N} 1_{[0,\infty )}^{(k)} (W_{t_{i-1}}) \Delta t
\big]^{2} \\
&=
\Big\{
\sum_{i=l}^{N}
\frac{ \sqrt{ (k-1)! } }{ (t_{i})^{ \frac{k-1}{2} } }
\mathbf{E} \big[
\delta_{0} (W_{t_{i}})
H_{k-1} \left( \frac{ W_{t_{i}} }{ \sqrt{t_{i}} } \right)
\big]
\Delta t
\Big\}^{2} \\
&=
\Big\{
\sum_{i=l}^{N}
\frac{ \sqrt{ (k-1)! } }{ (t_{i})^{ \frac{k-1}{2} } }
H_{k-1}(0) \frac{1}{ \sqrt{2\pi t_{i}} }
\Delta t
\Big\}^{2} \\
&=
k! \cdot \frac{ H_{k-1}(0)^{2} }{ 2\pi k }
\Big\{
\sum_{i=l}^{N}
\frac{ \Delta t }{ (t_{i})^{ n/2 } }
\Big\}^{2}.
\end{split}
\end{equation*}
By a similar argument, we find
$$
\mathbf{E} \big[
1_{[0,\infty )} (W_{T}) H_{k} \left( \frac{W_{T}}{\sqrt{T}} \right)
\big]
=
\frac{ H_{k-1}(0) }{ \sqrt{2\pi k} }
$$
and therefore
\begin{equation*}
\begin{split}
\Vert \mathrm{Err}_{N} \Vert_{L^{2}}^{2}
=
\sum_{k=2}^{\infty}
Z_{N,k}
\mathbf{E} \big[
1_{[0,\infty )} (W_{T}) H_{k} \left( \frac{W_{T}}{\sqrt{T}} \right)
\big]^{2}
\end{split}
\end{equation*}
where
\begin{equation}\label{ZNk-def}
\begin{split}
&
Z_{N,k}:=\sum_{l=1}^{N}
\Big\{
\sum_{i=l}^{N}
\frac{ \Delta t }{ (t_{i})^{ k/2 } }
\Big\}^{2}
\Big\{
( t_{l} )^{k} - ( t_{l-1} )^{k} - k ( \Delta t )( t_{l-1} )^{k-1}
\Big\} .
\end{split}
\end{equation}

On the other hand, by Lemma \ref{ZNk} below, we know that 
there exists a constant $ K > 0 $ such that
\begin{equation*}
Z_{N,k} \leq \frac{K}{N}
\end{equation*}
for each $k=2,3,\cdots$ and $N=3,4,\cdots$. 
Hence we have
\begin{equation*}
\Vert \mathrm{Err}_{N} \Vert_{L^{2}}^{2}
\leq
\frac{2K}{N}
\Vert 1_{[0,\infty )} (W_{T}) \Vert_{L^{2}}^{2}.
\end{equation*}

\end{proof}%%%%%%%%%%%%%%%%%%%%%%%%%%%%%%%%%%%%%%%%%%%%%%%%%

\begin{lemma}\label{ZNk}%%%%%%%%%%%%%%%%%%%%%%%%%%%%%%%%%%%%%%%%%%%%%%%
%There is a constant $K>0$ (which may depend on $T$) such that
For $ k \geq 2 $, it holds that 
\begin{equation}\label{est30}
Z_{N,k} \leq \frac{9 T^2}{N}. 
\end{equation}
where
$ Z_{N,k} $ is given as above in (\ref{ZNk-def}). 
\end{lemma}%%%%%%%%%%%%%%%%%%%%%%%%%%%%%%%%%%%%%%%%%%%%%%%%%
\begin{proof}%%%%%%%%%%%%%%%%%%%%%%%%%%%%%%%%%%%%%%%%%%%%%%%
We may write
\begin{equation*}
\begin{split}
&
Z_{N,k}
=
\sum_{ l=1 }^{N}
\Big[
\Big\{
\sum_{i=l}^{N}
\left( \frac{ t_{l} }{ t_{i} } \right)^{ k/2 } \Delta t
\Big\}^{2}
-
\Big\{
\sum_{i=l}^{N}
\left(
\frac{ t_{l-1} }{ t_{i} }
\right)^{ k/2 } \Delta t
\Big\}^{2}
\Big] \\
&\hspace{15mm}-
k
\sum_{ l=1 }^{ N }
\Big\{
\sum_{i=l}^{N}
\frac{ ( t_{l-1} )^{ (k-1)/2 }  }{ (t_{i})^{ k/2 } } \Delta t
\Big\}^{2}
\Delta t.
\end{split}
\end{equation*}
For $ l \geq 2 $, we have
\begin{equation*}
\begin{split}
&\Big\{
\sum_{i=l}^{N}
\left(
\frac{ t_{l-1} }{ t_{i} }
\right)^{ k/2 } \Delta t
\Big\}^{2} \\
& =
\Big\{
\sum_{i=l-1}^{N}
\left(
\frac{ t_{l-1} }{ t_{i} }
\right)^{ k/2 } \Delta t
\Big\}^{2}
-2
\sum_{i=l-1}^{N}
\left(
\frac{ t_{l-1} }{ t_{i} }
\right)^{ k/2 } ( \Delta t )^{2}
+
( \Delta t )^{2} ,
\end{split}
\end{equation*}
and therefore, 
\begin{equation*}
\begin{split}
& \sum_{ l=2 }^{N}
\Big[
\Big\{
\sum_{i=l}^{N}
\left( \frac{ t_{l} }{ t_{i} } \right)^{ k/2 } \Delta t
\Big\}^{2}
-
\Big\{
\sum_{i=l}^{N}
\left(
\frac{ t_{l-1} }{ t_{i} }
\right)^{ k/2 } \Delta t
\Big\}^{2}
\Big] \\
&= \sum_{ l=2 }^{N}
\Big[
\Big\{
\sum_{i=l}^{N}
\left( \frac{ t_{l} }{ t_{i} } \right)^{ k/2 } \Delta t
\Big\}^{2}
-
\Big\{
\sum_{i=l-1}^{N}
\left(
\frac{ t_{l-1} }{ t_{i} }
\right)^{ k/2 } \Delta t
\Big\}^{2}\Big] \\
&\hspace{15mm} + 2 \sum_{ l=2 }^{N}
\sum_{i=l-1}^{N}
\left(
\frac{ t_{l-1} }{ t_{i} }
\right)^{ k/2 } ( \Delta t )^{2}
- N ( \Delta t )^{2}.
\end{split}
\end{equation*}
Using this, 
\begin{equation}
\begin{split}
&
Z_{N,k}
=%%%%%%%%
(\Delta t)^{2}
\if -
\Big\{
\sum_{i=1}^{N}
\left(
\frac{ t_{1} }{ t_{i} }
\right)^{ k/2 } \Delta t
\Big\}^{2} \fi
+
2
\sum_{ l=2 }^{ N } \sum_{i=l-1}^{N}
\left(
\frac{ t_{l-1} }{ t_{i} }
\right)^{ n/2 } ( \Delta t )^{2} \\
&\hspace{15mm}-
N ( \Delta t )^{2}
-
k
\sum_{ l=2 }^{ N}
\Big\{
\sum_{i=l}^{N}
\frac{ ( t_{l-1} )^{ (k-1)/2 }  }{ (t_{i})^{ k/2 } } \Delta t
\Big\}^{2}
\Delta t \\
&\leq%%%%%
2
\sum_{ l=2 }^{ N } \sum_{i=l-1}^{N}
\left(
\frac{ t_{l-1} }{ t_{i} }
\right)^{ k/2 } ( \Delta t )^{2}
-
k
\sum_{ l=1}^{ N }
\Big\{
\sum_{i=l}^{N}
\frac{ ( t_{l-1} )^{ (k-1)/2 }  }{ (t_{i})^{ k/2 } } \Delta t
\Big\}^{2}
\Delta t.
\label{EST-Z}
\end{split}
\end{equation}

We observe that
$$\displaystyle
2 \sum_{l=2}^{N} \sum_{i={l-1}}^{N}
\left( \frac{ t_{l-1} }{ t_{i} } \right)^{k/2}
( \Delta t )^{2}
$$
behaves like
$$\displaystyle
2 \int_{0}^{T} \hspace{-2mm} \int_{t}^{T}
\left( \frac{t}{s} \right)^{k/2} \hspace{-2mm} ds dt
$$
and
$$\displaystyle
k \sum_{l=1}^{N} \Big\{ \sum_{i=l}^{N}
\frac{ ( t_{l-1} )^{ (k-1)/2 } }{ ( t_{i} )^{k/2} } \Delta t
\Big\}^{2} \Delta t
$$
behaves like
$$\displaystyle
k \int_{0}^{T} \Big\{
\int_{t}^{T} \hspace{-1mm}
\frac{ t^{ (k-1)/2 } }{ s^{k/2} } ds
\Big\}^{2} dt
$$
as $N\to\infty$ respectively. We note that 
$$
2 \int_{0}^{T} \hspace{-2mm} \int_{t}^{T} \hspace{-2mm}
\left( \frac{t}{s} \right)^{k/2} \hspace{-2mm} ds dt
=
n \int_{0}^{T} \hspace{-2mm} \Big\{
\int_{t}^{T} \hspace{-1mm}
\frac{ t^{ (k-1)/2 } }{ s^{k/2} } ds
\Big\}^{2} dt
=
\left\{ \begin{array}{ll}
\displaystyle \frac{ T^{2} }{2} & \text{if $k=2$,} \\
\displaystyle \frac{ 2T^{2} }{k+2} & \text{if $k\geq 2$.}
\end{array}\right.
$$
Based on the observations, 
we estimate $Z_{N,k}$ by separating it into two terms;
\begin{equation*}
Z_{N,k} \leq Z_{N,k}^{1} + Z_{N,k}^{2}
\end{equation*}
where
\begin{equation*}
\begin{split}
Z_{N,k}^{1}
&:=
2
\sum_{ l=2 }^{ N } \sum_{i=l-1}^{N}
\left(
\frac{ t_{l-1} }{ t_{i} }
\right)^{ k/2 } ( \Delta t )^{2}
-
2 \int_{0}^{T} \hspace{-2mm} \int_{t}^{T} \hspace{-2mm}
\left( \frac{t}{s} \right)^{k/2} \hspace{-2mm} ds dt, \\
%%%%%%%%%%
Z_{N,k}^{2}
&:=
k \int_{0}^{T} \hspace{-2mm} \Big\{
\int_{t}^{T} \hspace{-1mm}
\frac{ t^{ (k-1)/2 } }{ s^{k/2} } ds
\Big\}^{2} dt
-
k
\sum_{ l=1 }^{ N }
\Big\{
\sum_{i=l}^{N}
\frac{ ( t_{l-1} )^{ (k-1)/2 }  }{ (t_{i})^{ k/2 } } \Delta t
\Big\}^{2}
\Delta t
\end{split}
\end{equation*}

We estimate each of them. Firstly, we have
\begin{equation}\label{est20}
\begin{split}
Z_{N,n}^{1}
&\leq%%%%%%%%
2 \sum_{l=2}^{N} \sum_{i=l-1}^{N-1}
\int_{ t_{l-2} }^{ t_{l-1} } \int_{ t_{i} }^{ t_{i+1} }
\Big\{
	\left( \frac{ t_{l-1} }{ t_{i} } \right)^{k/2}
	-
	\left( \frac{t}{s} \right)^{k/2}
\Big\}
ds dt \\
& \hspace{2cm} + 2 \sum_{l=2}^{N} \left( \frac{ t_{l-1} }{ t_{N} } \right)^{k/2}
(\Delta t)^2 \\
&\leq%%%%%%%%
2 \sum_{l=2}^{N} \sum_{i=l-1}^{N-1}
\int_{ t_{l-2} }^{ t_{l-1} } \int_{ t_{i} }^{ t_{i+1} }
\Big\{
	\left( \frac{ t_{l-1} }{ t_{i} } \right)^{k/2}
	-
	\left( \frac{ t_{l-2} }{ t_{i+1} } \right)^{k/2}
\Big\}
ds dt \\
& \hspace{2cm} + 2 \sum_{l=2}^{N} 
\left( \frac{ t_{l-1} }{ t_{N} } \right)^{k/2} \\
%& \hspace{2cm} + 2 \frac{T^2}{N} 
%\sum_{l=1}^{N-1} \left(\frac{l}{N}\right)^{k/2}  \frac{1}{N} \\ &\leq
%\frac{ 4 T^{2} }{ N }, 
&= 2 (\Delta t)^2 \sum_{l=2}^{N} \sum_{i=l-1}^{N-1}\left\{ \left( \frac{l-1}{i} \right)^{k/2} - \left( \frac{l-2}{i} \right)^{k/2} \right\} \\
&+ 2 (\Delta t)^2 \sum_{l=2}^{N} \sum_{i=l-1}^{N-1} \left\{
\left( \frac{l-2}{i} \right)^{k/2} 
-\left( \frac{l-2}{i+1} \right)^{k/2}\right\} \\
& \hspace{2cm} + 2 \sum_{l=2}^{N} 
\left( \frac{ t_{l-1} }{ t_{N} } \right)^{k/2}
(\Delta t)^2.
\end{split}
\end{equation}
By a bit of algebra, the last term in (\ref{est20}) is seen to be
\begin{equation}\label{est21}
2 (\Delta t)^2\sum_{l=2}^N 
\left\{ 1 + \left(\frac{l-1}{l}\right)^{k/2} \right\},
\end{equation}
which is bounded above by $4T^2/N $. 

Next, we estimate $ Z_{N,k}^2 $. 
We set
\begin{equation*}
\begin{split}
& I= 
\sum_{l=1}^{N} \int_{ t_{l-1} }^{ t_{l} } t^{k-1}
\Big\{
\int_{ t }^{ T } \hspace{-1mm}
\frac{ ds }{ s^{k/2} } 
\Big\}^{2} dt  \\
& \hspace{3cm} -
\sum_{l=1}^{N} \int_{ t_{l-1} }^{ t_{l} } \hspace{-2mm}
t^{k-1} 
\Big\{ \sum_{i=l}^{N}
\frac{ \Delta t }{ ( t_{i} )^{k/2} } 
\Big\}^{2} dt
\end{split}
\end{equation*}
and
\begin{equation*}
\begin{split}
& II = 
\sum_{l=1}^{N} \int_{ t_{l-1} }^{ t_{l} } \hspace{-2mm}
t^{k-1} 
\Big\{ \sum_{i=l}^{N}
\frac{ \Delta t }{ ( t_{i} )^{k/2} } 
\Big\}^{2} dt \\
& \hspace{3cm} 
- \sum_{l=1}^{N} \int_{ t_{l-1} }^{ t_{l} } \hspace{-2mm}
(t_{l-1})^{k-1} 
\Big\{ \sum_{i=l}^{N}
\frac{ \Delta t }{ ( t_{i} )^{k/2} } 
\Big\}^{2} dt.
\end{split}
\end{equation*}
Note that $ Z_{N,k}^{2} = k (I+II) $. 
For $ t_{l-1} \leq t \leq t_l $, $ l=1,\cdots, N $, we have
\begin{equation}\label{lowerest}
\begin{split}
&\int_{ t }^{ T } \hspace{-1mm}
\frac{ ds }{ s^{k/2} } -  \sum_{i=l}^{N}
\frac{ \Delta t }{ ( t_{i} )^{k/2} } \\
&= \sum_{i=l+1}^N \int_{t_{i-1}}^{t_i} \left(\frac{1}{s^{k/2}}
- \frac{1}{(t_i)^{k/2}} \right)\,ds 
+ \int_t^{t_l} \frac{ ds }{ s^{k/2} }
- \frac{\Delta t}{(t_l)^{k/2}} \geq 0,
\end{split}
\end{equation}
and 
\begin{equation*}
\begin{split}
& \sum_{i=l+1}^N \int_{t_{i-1}}^{t_i} \left(\frac{1}{s^{k/2}}
- \frac{1}{(t_i)^{k/2}} \right)\,ds \\
& \hspace{2cm}
\leq \sum_{i=l+1}^N \int_{t_{i-1}}^{t_i} \left(\frac{1}{(t_{i-1})^{k/2}}
- \frac{1}{(t_i)^{k/2}} \right)\,ds \\
& \hspace{2cm} = \Delta t \left( \frac{1}{(t_l)^{k/2}} - 
\frac{1}{(t_N)^{k/2}}
\right).
\end{split}
\end{equation*}
Combining these two, 
we have
\begin{equation*}
\begin{split}
& \left\{ \int_{ t }^{ T } \hspace{-1mm}
\frac{ ds }{ s^{k/2} } \right\}^2
-
\left\{ \sum_{i=l}^{N}
\frac{ \Delta t }{ ( t_{i} )^{k/2} }
\right\}^2 \\
& 
\leq \int_t^{t_l} \frac{ ds }{ s^{k/2} } \left(
\int_{ t }^{ T } \hspace{-1mm}
\frac{ ds }{ s^{k/2} } + \sum_{i=l}^{N}
\frac{ \Delta t }{ ( t_{i} )^{k/2} }
\right) 
\leq 2 \int_t^{t_l} \frac{ ds }{ s^{k/2} } \left(
\int_{ t }^{ T } \hspace{-1mm}
\frac{ ds }{ s^{k/2} }
\right) \\
&=
\left\{\begin{array}{ll}
\displaystyle
\frac{4}{k-2} (t^{1-\frac{k}{2}}- T^{1-\frac{k}{2}}) \int_t^{t_l} \frac{ ds }{ s^{k/2} } 
\leq
\frac{4}{k-2} t^{1-\frac{k}{2}}\int_t^{t_l} \frac{ ds }{ s^{k/2} }
&
\text{if $k\geq 3$,} \\
\displaystyle
2 \int_{t}^{t_{l}} \frac{ds}{s} \log \frac{T}{t}
&
\text{if $k=2$.}
\end{array}\right.
\end{split}
\end{equation*}
Then for $k\geq 3$,
\begin{equation}\label{estI}
\begin{split}
I
&\leq
\frac{4}{k-2}  
\sum_{l=1}^{N} \int_{t_{l-1}}^{t_l} \int_t^{t_l} 
\left( \frac{t}{s} \right)^{k/2}\,dsdt
\\
&\leq
\frac{4}{k-2}
\sum_{l=1}^{N} \int_{t_{l-1}}^{t_l} \int_t^{t_l} \hspace{-2mm} dsdt
= \frac{2}{k-2} \sum_{l=1}^{N} (t_l - t_{l-1})^2 
= \frac{2}{k-2} \frac{T^2}{N}
\end{split}
\end{equation}
and for $k=2$, we have
\begin{equation} \label{estI-2}
\begin{split}
I
&\leq
2 \sum_{l=1}^{N} \int_{ t_{l-1} }^{ t_{l} } \int_{t}^{t_{l}}
\frac{t}{s} \log \frac{T}{t} ds dt
\leq
2 \sum_{l=1}^{N} \int_{ t_{l-1} }^{ t_{l} }
\int_{t}^{t_{l}} \hspace{-2mm} ds
\hspace{1mm} \log \frac{T}{t} \hspace{1mm} dt \\
&\leq
2 \Delta t \sum_{l=1}^{N}
\Big\{
\Delta t \log T - \big[ t\log t - t \big]_{t=t_{l-1}+0}^{t=t_{l}}
\Big\}
=
\frac{2T^{2}}{N}.
\end{split}
\end{equation}

Now we turn to the estimate of $ II $.
By (\ref{lowerest}), for $k\geq 3$,
\begin{equation}\label{estII}
\begin{split}
II & \leq \sum_{l=1}^{N} \int_{ t_{l-1} }^{ t_{l} } \hspace{-2mm}
\left\{t^{k-1} - (t_{l-1})^{k-1} \right\}
 \left( \int_{ t }^{ T } \hspace{-1mm}
\frac{ ds }{ s^{k/2} } \right)^2 dt \\
& \leq \frac{4}{(k-2)^2}\sum_{l=1}^{N} \int_{ t_{l-1} }^{ t_{l} } \hspace{-2mm}
\left\{t^{k-1} - (t_{l-1})^{k-1} \right\}
t^{2-k}
 dt \\
&= \frac{4(k-1)}{(k-2)^2}
\sum_{l=1}^{N} \int_{ t_{l-1} }^{ t_{l} }
\int_{t_{l-1}}^t \left( \frac{s}{t}
\right)^{k-2}
ds dt \\
&\leq \frac{4(k-1)}{(k-2)^2}
\sum_{l=1}^{N} \int_{ t_{l-1} }^{ t_{l} }
\int_{t_{l-1}}^t \hspace{-3mm}
ds dt = \frac{2(k-1)}{(k-2)^2}\frac{T^2}{N}.
\end{split}
\end{equation}
For $k=2$, we have
\begin{equation} \label{estII-2}
\begin{split}
II
&\leq
\sum_{l=1}^{N} \int_{ t_{l-1} }^{ t_{l} } \hspace{-2mm}
( t - t_{l-1} )
\left( \int_{ t }^{ T } \hspace{-1mm}
\frac{ ds }{ s } \right)^2 dt \\
&\leq
\Delta t \sum_{l=1}^{N} \int_{ t_{l-1} }^{ t_{l} }
\left( \log \frac{T}{t} \right)^2 dt
=
\frac{2T^{2}}{N}.
\end{split}
\end{equation}

By (\ref{estI}), (\ref{estI-2}), (\ref{estII}) and (\ref{estII-2}), we have 
\begin{equation}\label{est22}
Z_{N,k}^2 \leq \frac{5 T^2}{N}. 
\end{equation}
Combining (\ref{est21}) and (\ref{est22}), 
we obtained (\ref{est30}). 

\end{proof}%%%%%%%%%%%%%%%%%%%%%%%%%%%%%%%%%%%%%%%%%%%%%%%%%

\begin{rem}%%%%%%%%%%%%%%%%%%%%%%%%%%%%%%%%%%%%%%%%%%%%%%%%%
A result by Ngo-Ogawa (\cite{NgOg}, Theorem 2.2.) tells us that
the sequence of processes
$$
\Big\{
	n^{3/4}
	\Big(
	\frac{1}{N} \sum_{i=0}^{[Nt]} 1_{[0,\infty)} (X_{i/N})
	-
	\int_{0}^{t} \hspace{-2mm}
	1_{[0,\infty )} (X_{s}) ds
	\Big)
\Big\}_{t\geq 0}
$$
is tight for a diffusion $X=(X_{t})_{t\geq 0}$ although their results are more general.
Moreover they say that this is optimal in $L^{2}$-sense in the case where $X$ is the standard
Brownian motion (see \cite{NgOg}, Proposition 2.3).
\end{rem}%%%%%%%%%%%%%%%%%%%%%%%%%%%%%%%%%%%%%%%%%%%%%%%%%%%

%------------------------------------- COMMENTED OUT END ----------

%%%%%%%%%%%%%%%%%%%%%%%%%%%%%%%%%%%%%%%%%%%%%%%%%%%%%%%%%%%%
%\input{old-NonUnif.tex}*** OLD VERSION of the subsection ``Non-Uniform discretization case"
%\input{old-ErrWithEMAppr.tex}*** OLD VERSION of the subsection ASYMP
%\input{old-Appdx.tex}*** OLD VERSION of the subsection ``Appendix"
%\input{old-ActApprx.tex}*** OLD VERSION of the subsection ``Two Conceivable Finite-Dim. Apprx."
%\input{old-Poisson.tex}*** OLD VERSION of the subsection ``Poissonian Case"
%%%%%%%%%%%%%%%%%%%%%%%%%%%%%%%%%%%%%%%%%%%%%%%%%%%%%%%%%%%%
%%%%%%%%%%%%%%%%%%%%%%%%%%%%%%%%%%%%%%%%%%%%%%%%%%%%%%%%%%%%
%%%%%%%%%%%%%%% REFERENCES %%%%%%%%%%%%%%%%%%%%%%%%%%%%%%%%%
%%%%%%%%%%%%%%%%%%%%%%%%%%%%%%%%%%%%%%%%%%%%%%%%%%%%%%%%%%%%


\begin{thebibliography}{99}
\if
\bibitem{AOPU}{\sc Aase, K., Oksendal, B., Privault, N. and Uboe, J.}
``White noise generalizations of the Clark-Haussmann-Ocone theorem with application to mathematical finance", 
{\it Finance Stoch.} 4 (2000), no. 4, 465--496.

\bibitem{AAOsss}{\sc Akahori, J., Amaba, T., and 
Okuma, K.}, ``Some Simulation Results on the 
Computation of Delta of Path-Dependent Options Using 
a Discrete Version of Clark-Ocone Formula", 
to appear in {\em Proceedings of the 42nd ISCIE International Symposium on Stochastic Systems Theory and its Applications}, Inst. Syst. Control Inform. Engrs. (ISCIE), 2011. 

\bibitem{BM-Ouknine}{\sc Bahlali, K., Mezerdi, B., and Ouknine, Y.}
``A Haussmann-Clark-Ocone formula for functionals of diffusion processes with Lipschitz coefficients", 
{\em J. Appl. Math. Stochastic Anal.} 15 (2002), no. 4, 371--383. 
\fi

\bibitem{Bi}
{\sc Billingsley, P.}
``Convergence of probability measures". 
Second edition. Wiley Series in Probability and Statistics: Probability and Statistics. A Wiley-Interscience Publication. John Wiley \& Sons, Inc., New York, 1999. x+277 pp. ISBN: 0-471-19745-9

\if

\bibitem{Bouchard-Touzi}{\sc Bouchard, B. and Touzi, N.}
``Discrete time approximation and Monte Carlo simulation of backward stochastic differential equations", 
{\em Stochastic Processes and their Applications} 111 (2004) 175--206. 

\fi

\bibitem{BKL}{\sc Bertsimas, D., Kogan, L., and Lo, A.W.} 
``When Is Time Continuous?", {\em Journal of Financial Economics}, Vol 55. 
(2000), 173-204. 

\bibitem{clark}
{\sc Clark, J.M.C.}, ``The representation of functionals of Brownian motion by stochastic integral", {\it Ann.Math.Statist.} 41 (1970), 1282--1295.

\if

\bibitem{Cont}
{\sc Cont, R. and Fournie, D.A.}
``Functional Ito calculus and stochastic integral representation of martingales", arXiv:1002.2446 [math.PR]. 2010. 

\bibitem{ELK-Privault}{\sc El-Khatib, Y., and Privault, N.}
``Hedging in complete markets driven by normal martingales", 
{\it Appl. Math}. (Warsaw) 30 (2003), no. 2, 147--172.

\fi

\bibitem{GeGe}{Geiss, S. and Geiss C.}
``On approximation of a class of stochastic integrals and interpolation", 
{\em Stochastics and Stochastics Reports} 76 (2004) 339-362.

\if

\bibitem{Gobet-Labar}{\sc Gobet, E. and Labart, C.}
{\it Error expansion for the discretization of backward stochastic differential equations}, {\it Stochastic Process. Appl.} 117 (2007), no. 7, 803--829. 

\bibitem{Gobet-Azmi1}{\sc Gobet, E. and Makhlouf, A.}
``${\bf L}_2$-time regularity of BSDEs with irregular terminal functions", {\it Stochastic Process. Appl.} 120 (2010), no. 7, 1105--1132.

\bibitem{Gobet-Azmi2}{\sc Gobet, E. and Makhlouf, A.}
``The tracking error rate of the Delta-Gamma hedging strategy", 
to appear in {\it Mathematical Finance}

\fi

%\bibitem{Gobet-Higa}{\sc Gobet, E., and Kohatsu-Higa, A.}
%``Computation of Greeks for barrier and look-back options using Malliavin calculus", {\em Electron. Comm. Probab.} 8 (2003), 51--62. 

\bibitem{Gobet-Tenam}{\sc Gobet, E. Temam, E.} 
``Discrete time hedging errors for options with irregular pay-offs", 
{\em Finance and Stochastics} 5 (3) (2001) 357--367. 

%\bibitem{Haussmann1}{\sc Haussmann, U. G.}
%``Functionals of Ito processes as stochastic integrals",
%{\it SIAM J. Control Optimization} 16 (1978), no. 2, 252?269.

\bibitem{Hayashi-Mykland}{\sc Hayashi, T. and Mykland, P. A. }
``Evaluating hedging errors: an asymptotic approach",
{\em Math. Finance} 15 (2005), no. 2, 309--343.

\bibitem{Haussmann2}{\sc Haussmann, U. G.}
``On the integral representation of functionals of It\^o processes",
{\it Stochastics} 3 (1979), no. 1, 17--27.

\if

\bibitem{HNS}{\sc Hu, Y., Nualart, D. and Song, X.}
``Malliavin calculus for backward stochastic differential equations and application to numerical solutions", preprint.  

\fi

\bibitem{IW}{\sc Ikeda, N., and Watanabe, S.}
{\em Stochastic Differential Equations and Diffusion Processes}, 2nd eds.
North-Holland (1981). 

%\bibitem{KOL}{\sc Karatzas, I., Ocone, D. and Li, J.},
%``An extension of Clark's formula",  
%{\it Stochastics Stochastics Rep.} 37 (1991), no. 3, 127--131.

%\bibitem{Ki}
%{\sc Kitabeppu, Yu},
%{\em Coarse Ricci curvature on the space of probability measures},
%arXiv:1301.0978v1, Submitted on 6 Jan 2013.

%\bibitem{Ku}
%{\sc Kuwada, Kazumasa}(J-OCHGH)
%{\em Duality on gradient estimates and Wasserstein controls}. (English summary)
%J. Funct. Anal. 258 (2010), no. 11, 3758-3774.

\if

\bibitem{LM}
{\sc Lemaire, V.(F-PARIS6-PMA); Menozzi, S.(F-PARIS7-PMA)}
{\em On some non asymptotic bounds for the Euler scheme.} (English summary)
Electron. J. Probab. 15 (2010), no. 53, 1645-1681.

\fi

\bibitem{MaTh}{\sc Malliavin, Paul; Thalmaier, Anton}(F-POIT-DM)
{\em ``Stochastic calculus of variations in mathematical finance".}
Springer Finance. Springer-Verlag, Berlin, 2006. xii+142 pp. ISBN: 978-3-540-43431-3; 3-540-43431-3
91-02 (49J45 60H07 60H30 65C50 91B28)

\if

\bibitem{Nu}
{\sc Nualart, David}(E-BARUM)
``The Malliavin calculus and related topics". (English summary)
Probability and its Applications (New York). Springer-Verlag, New York, 1995. xii+266 pp. ISBN: 0-387-94432-X

\bibitem{Nuart-Scouten1}{\sc Nualart, D. and 
Schoutens, W.}
``Chaotic and predictable representations for Levy processes",
{\em Stochastic Process. Appl.} 90 (2000), no. 1, 109--122. 

\bibitem{Nuart-Scouten2}{\sc Nualart, D. and Schoutens, W.}
``Backward stochastic differential equations and Feynman-Kac formula for Levy processes, with applications in finance",
{\em Bernoulli} 7 (2001), no. 5, 761?776. 

\fi

\bibitem{NgOg}{\sc Ngo, Hoang-Long}(J-RITS2); {\sc Ogawa, Shigeyoshi}(J-RITS2)
\emph{``On the discrete approximation of occupation time of diffusion processes"}. (English summary)
Electron. J. Stat. 5 (2011), 1374-1393.

\bibitem{ocone}{\sc Ocone, D.}, ``Malliavin's calculus and stochastic integral representations of functionals of diffusion processes", {\it Stochastics} 12 (1984),no. 3--4, 161-185.

\if

\bibitem{ocone karatzas}{\sc Ocone, D. and Karatzas, I.}, 
``A generalized Clark representation formula, with application to optimal portfolios", {\it Stochastics Rep.} p.34 (1991), no. 3--4, 187-220. 

\fi

%\bibitem{Ol}
%{\sc Ollivier, Yann}(F-ENSLY-PM)
%``Ricci curvature of Markov chains on metric spaces." (English summary)
%J. Funct. Anal. 256 (2009), no. 3, 810-864.

\if

\bibitem{pedersen}{\sc Pedersen, J.}
``Convergence of strategies: an approach using Clark-Haussmann's formula", {\it Finance Stoch.} 3 (1999), no. 3, 323--344. 

\fi

\bibitem{Privault1}{\sc Privault, N.}
{\it Stochastic analysis in discrete and continuous settings with normal martingales}, Lecture Notes in Mathematics, 1982. Springer-Verlag, Berlin, 2009.

%\bibitem{Privault-Schouten}{\sc
%Privault, N., and Schoutens, W. }
%``Discrete chaotic calculus and covariance identities", 
%Stoch. Stoch. Rep. 72 (2002), no. 3-4, 289--315.

%\bibitem{martin}{\sc Leitz-Martini, M.}, ``A discrete Clark-Ocone formula"
%Maphysto Research Report No 29, (2000).

\bibitem{ReedSimon}
{\sc Reed, M.} and {\sc Simon, B.} 
{\em Methods of Modern Mathematical Physics I. Functional Analysis},
Academic Press, 1980. 

\bibitem{jean}{\sc Renaud,J.F.} and {\sc R\'emillard,B.},
``Explicit martingale representations for Brownian functionals and applications to option hedging", {\it Stochastic Analysis and Applications}, 25 (2007), 810--820. 

\bibitem{Roo}{\sc Rootzen, H.}
``Limit Distributions for the Error in Approximations 
of Stochastic Integrals", 
{\em Ann. Probab.} 
Volume 8, Number 2 (1980), 241-251.

\bibitem{Ru}{\sc Rudin, W.} {\em Functional Analysis} 
(2nd ed.), 1991, McGraw-Hill.

\bibitem{Temam}{\sc Temam, E.} 
``Analysis of error with Malliavin calculus: application to hedging", 
{\em Math. Finance} 13 (2003), no.1, 201--214.

%\bibitem{Vi}
%{\sc Villani, C\'edric}(F-ENSLY-PM)
%{\em Optimal transport. Old and new}.
%Grundlehren der Mathematischen Wissenschaften [Fundamental Principles of Mathematical Sciences], 338.
%Springer-Verlag, Berlin, 2009. xxii+973 pp. ISBN: 978-3-540-71049-3

\if

\bibitem{Zhang}{\sc Zhang, J.} 
``A numerical scheme for BSDEs", 
{\it Ann. Appl. Probab.} 14 (2004), no. 1, 459--488.

\fi
\end{thebibliography}
\end{document}